%% file: PEP2_siamTemp.tex
\documentclass{siamart}
\usepackage{amsmath}
\numberwithin{theorem}{section}
\usepackage{booktabs}
\usepackage{multirow}
\usepackage{fixltx2e}

\input{ex_shared}
\input{preamble}

\newcommand{\revd}[1]{{#1}}
\newcommand{\rev}[1]{{#1}}
\newcommand{\revv}[1]{{#1}}
\newcommand{\del}[1]{#1}
\begin{document}

\begin{center}
\fbox{
	\parbox{12cm}{
		\setlength{\parskip}{.5cm}
		Published in {\bfseries SIAM Journal on Optimization}, DOI 10.1137/16M108104X.
		{Compared to \href{https://epubs.siam.org/doi/abs/10.1137/16M108104X}{SIAM's version}, one typo corrected in {\bfseries\color{red}red} in Theorem~\ref{thm:noncvxinterp}.}
	}}\vspace{2cm}
\end{center}

\maketitle

\begin{abstract}
We provide a framework for computing the exact worst-case performance of any algorithm belonging to a broad class of oracle-based first-order methods for composite convex optimization, including those performing explicit, projected, proximal, conditional, and inexact (sub)gradient steps. We simultaneously obtain tight \rev{worst-case guarantees} and explicit \revv{instances of optimization} problems on which the algorithm reaches this worst-case. We achieve this by reducing the computation of the worst-case to solving a convex semidefinite program, generalizing previous works on performance estimation by Drori and Teboulle~\cite{Article:Drori} and the authors~\cite{taylor2015smooth}.

We use these developments to obtain a tighter analysis of the proximal point algorithm and of several variants of fast proximal gradient, conditional gradient, subgradient, and alternating projection methods. In particular, we present a new analytical worst-case guarantee for the proximal point algorithm that is twice better than previously known and improve the standard \rev{worst-case} guarantee for the conditional gradient method by more than a factor of two.

\del{We also show how the optimized gradient method proposed by Kim and Fessler~\cite{kim2014optimized} can be extended by incorporating a projection or a proximal operator, which leads to an algorithm that converges in the worst-case twice as fast as the standard accelerated proximal gradient method~\cite{beck2009fast}.}
\end{abstract}


\section{Introduction}
Consider the composite convex minimization problem
\begin{equation}
\min_{x\in \E} \left\{F(x)\equiv \sum_{k=1}^{n} F^{(k)}(x) \right\},\tag{CM}\label{eq:origOpt}
\end{equation}where $\E$ is a finite-dimensional real vector space and each functional component\break $F^{(k)}:\fdef$ is a convex function belonging to some class $\mathcal{F}_k(\E)$---e.g., smooth or nonsmooth, strongly convex or not, indicator functions---for which some operations are assumed to be available in closed form (e.g., computing a gradient, projecting on the domain, computing a proximal step).

We are interested in the composite optimization problem~\eqref{eq:origOpt} because it naturally allows representing and exploiting a lot of the structure in many problems, which can play a major role in our ability to efficiently solve them (see~\cite{nesterov2007gradient} among others). In addition, the class of composite convex optimization problems arises very commonly in practice, as it contains, for example, constrained, $\ell_1$- and $\ell_2$-regularized convex optimization problems.

We focus on black-box oracle-based algorithms that use first-order information to approximately solve~\eqref{eq:origOpt} and in particular on obtaining \revv{exact and global} worst-case guarantees \revv{on} their performances. That is, for a given algorithm, we simultaneously seek to obtain \rev{worst-case guarantees}---for example, on objective function accuracy---and an instance of~\eqref{eq:origOpt} for which the algorithm behaves as such. In this work, we investigate fixed-step linear first-order methods (FSLFOM), which include among others fixed-step projected, proximal, conditional, and inexact (sub)gradient methods.

This work builds on the recent idea of performance estimation, first developed by Drori and Teboulle in~\cite{Article:Drori} and followed up on by Kim and Fessler~\cite{kim2014optimized} and the authors~\cite{taylor2015smooth}. The approach was initially tailored for obtaining upper bounds on the worst-case behavior of fixed-step gradient methods for unconstrained minimization of a single smooth convex objective function. Motivated by subsequent results (see among others~\cite{kim2015convergence,kim2014optimized}) we extend the framework of performance estimation to the composite case involving a much broader class of algorithms and function classes (see \secref{sec:priorwork} for more details about previous works).

Our performance estimation framework relies on formulating the \revv{worst-case computation} problem as \revv{a tractable semidefinite program (SDP)}, which can be tackled with standard solvers~\cite{Article:Yalmip,Article:Mosek,Article:Sedumi}. It enjoys the following attractive features:
\begin{itemize}
\item Any primal feasible solution to this SDP leads to a lower bound on the worst-case performance of the method under consideration, by exhibiting a particular instance of~\eqref{eq:origOpt}.
\item Any dual feasible solution to this SDP corresponds to an upper bound on the worst-case performance of the method under consideration, which can be converted into an explicit proof based on a combination of valid inequalities.
\end{itemize}

\subsection{Notation}\label{subsec:notations}In this paper, we work in a finite-dimensional real vector space $\E$ and the corresponding dual space $\Es$ consisting of all linear functions on~$\E$, and denote their dimension by $d=\dim \E=\dim \Es$. We consider a dual pairing\footnote{The dual pairing is a real bilinear map $\inner{.}{.}:\Es\times\E\rightarrow \R$ satisfying (i)~$\forall x\in \E\backslash \left\{0\right\}$, $\exists~s\in~\Es~\text{ such that } \inner{s}{x}\neq 0,$ and (ii)~$\forall s\in \Es\backslash \left\{0\right\}, \exists x\in\E \text{ such that } \inner{s}{x}\neq 0$.} between those spaces, denoted by $\inner{.}{.}:\Es\times\E\rightarrow \R$. We also consider a self-adjoint positive definite\footnote{That is, a linear operator $B$ satisfying (i) $\innerS{Bx}{y}=\innerS{By}{x} \ \forall x,y\in\E$ (self-adjoint), and (ii)~$\inner{Bx}{x}>0 \ \forall x\in\E \backslash \left\{0\right\}$ (positive definite). \revd{A direct consequence of $B$ satisfying those assumptions is the existence of the linear operator $B^{-1}$.}} linear operator $B:\E\rightarrow \Es$ for $\inner{.}{.}$, which allows defining the following primal and dual norms:
\begin{align*}
\normpsq{x}=\inner{Bx}{x} \ \forall x\in\E, \quad \normdsq{s}=\inner{s}{B^{-1}s} \ \forall s\in\Es.
\end{align*}
We denote $\innerE{x}{y}=\inner{Bx}{y}$ for $x,y\in\E$ and $\innerEs{x}{y}=\inner{x}{B^{-1}y}$ for $x,y\in\Es$. The usual case is simply $\E=\Es=\R^d$ with $\inner{x}{y}=x^\top\! y$ the standard Euclidean inner product and $B$ the identity operator, for which we also have $\normpsq{x}=\normdsq{x}=\inner{x}{x}$.

\revd{In addition, we use the notation $\mathcal{F}_{0,\infty}$ for the set of closed, proper, and convex functions.}
For a convex function $f: \fdef$, we denote by $f^*:\fsdef$ its Legendre--Fenchel conjugate
\[ f^*(y)=\sup_{x\in  \E} \inner{y}{x}-f(x),\] by $\partial f(x)$ the subdifferential of $f$ at $x$ (set of all subgradients of $f$ at $x$), and by $\tilde{\nabla} f(x)$ a particular subgradient of $f$ at $x$. Similarly, the gradient of a differentiable function $f$ at $x$ is denoted by $\nabla f(x)$. 

For notational convenience we denote by $K=\left\{1,\hdots,n\right\}$ the set of indices corresponding to the different components $F^{(k)}$ in the objective function of~\eqref{eq:origOpt}. We also denote by $\mathcal{F}_K(\E)$ the set of functions of the form~\eqref{eq:origOpt} with components $F^{(k)}\in \mathcal{F}_k(\E)$ $\forall k\in K$---that is, $F\in\mathcal{F}_K(\E)$.

Finally, we use the standard notation $e_i$ for the unit vector having a single $1$ as its $i$th component.
\subsection{Performance estimation problems}
In~\cite{taylor2015smooth}, we introduced a formal definition for the performance estimation problem in the case of a black-box first-order method for unconstrained minimization of a single convex function $F$. We now generalize the performance estimation framework for handling multiple components in the objective function.

First, we formalize black-box methods using the concept of \emph{black-box oracles}. That means that methods are only allowed to access the different components of the objective function by calling some routines, or oracles, returning some information about them at a given point. In particular, we focus on the standard first-order oracle for $F^{(k)}$: $\mathcal{O}_{F^{(k)}}(x)=\left(F^{(k)}(x),\tilde{\nabla} F^{(k)}(x)\right)$, where $\tilde{\nabla} F^{(k)}(x)\in\partial F^{(k)}(x)$ is a subgradient of $F^{(k)}$ at $x$. The general formalism of the approach is nevertheless also valid for other standard oracles, as, for example, zeroth-order  or second-order ones---that is, $\mathcal{O}_{F^{(k)}}(x)=\left(F^{(k)}(x)\right)$ or $\mathcal{O}_{F^{(k)}}(x)=\left(F^{(k)}(x),\nabla F^{(k)}(x),\nabla^2 F^{(k)}(x)\right)$. However, as we will see, our ability to solve the corresponding performance estimation problems in an exact way is currently limited to first-order oracles.

Second, we consider a sequence of $N+1$ iterates $\rev{\left\{x_i\right\}_{0 \le i\le N}} \subset \E$, corresponding to a method that performs $N$ steps from an initial iterate $x_0$. For each of those iterates we consider the set oracle calls for each functional component\footnote{That is, we chose to associate a call to each oracle to every iterate. This is mostly for notational convenience and does not induce any loss of generality, as a method can always avoid using the information returned by one of the oracles at some iterations.} $\mathcal{O}_{F^{(k)}}$: \rev{ $\left\{\mathcal{O}_{F^{(k)}}(x_i)\right\}_{0 \le i\le N}$}. 

Third, we consider a method $\mathcal{M}$ whose iterates can be computed by combining past and current oracle information about $F$. This means that after the method has performed \modAT{$i-1$} steps, the  next iterate \modAT{$x_{i}$ should be computable as a} solution to an equation of the form
\begin{equation} 
\textsc{Equation}(x_0,\left\{\mathcal{O}_{F^{(k)}}(x_{0})\right\}_{k\in K},x_1,\left\{\mathcal{O}_{F^{(k)}}(x_{1})\right\}_{k\in K}, \hdots,\modAT{x_{i}},\left\{\mathcal{O}_{F^{(k)}}\modAT{(x_{i})}\right\}_{k\in K}).\label{EqImplicite} \tag{EQ\textsubscript{i}} 
\end{equation}Note that the only unknown in this equation is \modAT{$x_{i}$}, and that it thus provides an implicit definition for the next step. We will see later that this assumption on $\mathcal{M}$ includes a large number of existing methods for composite optimization.

Finally, we consider a real-valued performance criterion $\mathcal{P}$ \revd{ for evaluating the efficiency of the method}. \revd{In what follows sequel, we assume without loss of generality that the lower the value of $\mathcal{P}$, the better the corresponding method.} 

In our framework, this performance criterion is \revd{generally} allowed to depend on information returned by the oracles $\mathcal{O}_{F^{(k)}}$ at all the iterates \rev{$\left\{x_i\right\}_{0 \le i\le N}$}, but also at an extra point $x_* \in \E$ assumed to be an optimal solution to problem~\eqref{eq:origOpt}. \revd{Also, we} allow $\mathcal{P}$ to depend on \revd{the} iterates themselves. Examples of such performance criteria include objective function accuracy $F(x_N)-F(x_*)$,  and distance to an optimal solution $\normpsq{x_N - x_*}$. For notational convenience we introduce an index set for all iterates (including optimal solution) $I = \{ 0, 1, \ldots, N, * \}$.

The worst-case performance of method $\mathcal{M}$ on~\eqref{eq:origOpt} is then the optimal value of the following optimization problem, with both functions $\left\{ F^{(k)} \right\}_{k \in K}$ and iterates $\left\{x_i\right\}_{i\in I}$ as variables, which we call a performance estimation problem (PEP):
\begin{align}
&\sup_{\left\{ F^{(k)} \right\}_{k \in K},\left\{x_i\right\}_{i\in I}} \ \mathcal{P}(\rev{\left\{\mathcal{O}_{F^{(k)}}(x_i)\right\}_{i\in I,k\in K}},\left\{x_i\right\}_{i\in I})  \tag{PEP}\label{Intro:PEP} \\
\text{\revd{subject to} }& F^{(k)}\in \mathcal{F}_k(\E) \ \forall k\in K,  \notag\\
& x_0 \text{ satisfies some initialization condition,} \notag \\
& x_{i} \text{ is computed by $\mathcal{M}$ according to \eqref{EqImplicite}  $\forall 1 \le i \le N$,} 
\notag\\
& x_* \text{ is a minimizer of } F(x). \notag
\end{align}That is, a solution to~\eqref{Intro:PEP} corresponds to an instance of problem~\eqref{eq:origOpt} on which method $\mathcal{M}$ behaves as badly as possible with respect to the performance criterion~$\mathcal{P}$. The initialization condition on $x_0$ is required as most methods exhibit unbounded worst-case performance without it. In what follows we will mostly restrict ourselves to the classical approach, which consists in bounding the initial distance to an optimal solution with a constant $R$, i.e., assume $\normp{x_0-x_*} \leq R$.

Note that~\eqref{Intro:PEP} is inherently an infinite-dimensional optimization problem, as functions $F^{(k)}$ appear as variables. However, a crucial observation is that, due to the black-box assumption on the objective components, this problem can be cast completely equivalently in a finite-dimensional fashion. Indeed, introducing the \emph{outputs} of the oracle calls as variables, namely, $O^{(k)}_i = \mathcal{O}_{F^{(k)}}(x_i)$ for all iterates $i \in I$ and oracles $k \in K$, we observe that steps of method $\mathcal{M}$ can be still be computed using only information contained in variables $O^{(k)}_i $, so that we can reformulate \eqref{Intro:PEP} as 
\begin{align*}
&\sup_{\left\{O^{(k)}_i \right\}_{i\in I, k \in K},\left\{x_i\right\}_{i \in I}} \ \mathcal{P}\left(\left\{O^{(k)}_i \right\}_{i\in I, k \in K}, \left\{x_i\right\}_{i\in I}\right), \tag{PEP2}\label{Intro:PEP2} \\
\text{\revd{subject to} } 
&\text{$\exists F^{(k)} \in \mathcal{F}_k(\E)$ satisfying } \mathcal{O}_{F^{(k)}}(x_i) = O^{(k)}_i 
\ \forall i \in I, k\in K, \\
& x_0 \text{ satisfies some initialization condition,} \notag \\
& x_{i} \text{ is computed by $\mathcal{M}$ according to \eqref{EqImplicite} $\forall 1 \le i \le N$,} 
\notag\\
& x_* \text{ is a minimizer of } F(x). \notag
\end{align*}
Note the central role played by the interpolation conditions $\mathcal{O}_{F^{(k)}}(x_i) = O^{(k)}_i$  $\forall i\in I$ and $k\in K$, which enforce the existence of functions $F^{(k)}$ compatible with the output of the oracles. In the next subsection we describe situations for which this formulation is tractable.

\subsection{First-order methods and first-order convex interpolation} In the remainder of this work, we restrict ourselves to first-order oracles and methods. We now investigate the concept of (first-order) convex interpolability, in order to make existence constraints from~\eqref{Intro:PEP2} tractable---more precise requirements are detailed in \secref{sec:prox-PEP}. From the assumptions, the existence constraint for function $F^{(k)}$
\[ \exists F^{(k)} \in \mathcal{F}_k(\E) \text{ satisfying } \mathcal{O}_{F^{(k)}}(x_i) = O^{(k)}_i \, \forall i \in I,\]
found in~\eqref{Intro:PEP2}, may be expressed in terms of first-order information only. Considering oracles returning first-order information  $\mathcal{O}_{F^{(k)}}(x)=(F^{(k)}(x), \tilde{\nabla} F^{(k)}(x))$, we denote their output at point $x_i$  by $\mathcal{O}_{F^{(k)}}(x_i)=O^{(k)}_i = (f_i^{(k)},g_i^{(k)})$. The above existence constraint can be rephrased into the set of interpolation conditions 
\begin{equation} \tag{INT}
\exists F^{(k)} \in \mathcal{F}_k(\E) \text{ satisfying } F^{(k)}(x_i) = f^{(k)}_i \text{ and } g^{(k)}_i\in \partial F^{(k)}(x_i), \label{def:cvx_interp_cond} 
\end{equation} 
which leads us to introduce the following general definition.
\begin{definition}[\revd{$\mathcal{F}(\E)$-interpolation}] Let $I$ be an index set and $\mathcal{F}(\E)$ a class of convex functions, and consider the set of triples $S = \left\{(x_i,g_i,f_i)\right\}_{i\in I}$ where  $x_i\in\E$, $g_i\in\Es$ and $f_i\in\mathbb{R}$ $\forall i \in I$. The set $S$ is $\mathcal{F}(\E)$-interpolable if and only if there exists a function $F\in\mathcal{F}(\E)$ such that both $g_i\in\partial F(x_i)$ and $F(x_i)=f_i$ hold $\forall i\in I$. 
\label{def:CvxComp}
\end{definition}

The notion of $\mathcal{F}(\E)$-interpolation can be considered for any class of convex functions. 
 It allows us to formulate our PEP in its final form,
\begin{align*}
&\sup_{\left\{(f_i^{(k)},g_i^{(k)})\right\}_{i\in I, k \in K},\left\{x_i\right\}_{i \in I}} \ \mathcal{P}\left(\left\{(f_i^{(k)},g_i^{(k)})\right\}_{i\in I, k \in K}, \left\{x_i\right\}_{i\in I}\right), \tag{f-PEP}\label{Intro:dPEP}\\
\text{\revd{subject to} }
&\left\{(x_i,g^{(k)}_i,f^{(k)}_i)\right\}_{i\in I} \text{ is } \mathcal{F}_k\text{-interpolable } \forall k \in K, \\
& x_0 \text{ satisfies some initialization condition,} \notag \\
& x_{i} \text{ is computed by $\mathcal{M}$ according to \eqref{EqImplicite}  $\forall 1 \le i \le N$,} \notag\\
& x_* \text{ is a minimizer of } F(x). \notag
\end{align*}
We conclude that identifying explicit conditions for convex interpolability by a given class of functions will be the key to eliminate the infinite-dimensional functional variables from~\eqref{Intro:PEP} and transform it into a tractable estimation problem.

First-order convex interpolation was originally developed in~\cite{taylor2015smooth} for classes of (possibly) $L$-smooth and (possibly) $\mu$-strongly convex functions. In \secref{sec:cvx_interp}, we extend these results to classes of functions involving simultaneously strong convexity, smoothness, gradient boundedness, and domain boundedness (for different norms). Those extensions also allow us to consider interpolation by indicator or support functions, which may among others be used for problems involving constraints.

Also, note that the notion of first-order interpolability can be adapted for nonconvex functions as well. Replacing the concept of subdifferentiability by standard differentiability can be used to study the convergence of first-order algorithms in the cases where some components $F^{(k)}$ are not convex (see \secref{sec:smoothnconvex}). 

\subsection{Prior work}\label{sec:priorwork}
The concept of performance estimation showed itself very promising in the pioneer work of Drori and Teboulle~\cite{Article:Drori} and later in the work of Kim and Fessler~\cite{kim2014optimized}. In their work~\cite{Article:Drori}, Drori and Teboulle proposed a convex relaxation to obtain numerical upper bounds on the worst-case behavior of fixed-step first-order algorithms minimizing a {single smooth convex function} over ${\R^d}$, which turned out to be tight in surprisingly many situations.\footnote{\modAT{An extension to provide upper bounds for the fixed-step projected gradient method is also provided in Drori's {Ph.D.} thesis~\cite{drori2014contributions}.}} They also proposed a way to numerically optimize the step size parameters of a fixed-step algorithm \modAT{by} minimizing an upper bound on its worst-case. Their approach is based on {semidefinite relaxations} of~\eqref{Intro:PEP} and was taken further by Kim and Fessler~\cite{kim2014optimized}, who derived analytically the optimized gradient method (OGM) previously identified numerically by Drori and Teboulle.

The performance estimation approach on the same smooth unconstrained minimization is further studied in~\cite{taylor2015smooth}, where {convex interpolation} allows the derivation of an exact convex reformulation of the problem, leading to tight worst-case estimates. The obtained semidefinite formulation also forms the basis for this work. 

Another recent and closely related approach for studying performances of first-order methods consists in viewing optimization algorithms as dynamical systems and to use the related stability theory in order to numerically analyze them. This idea is proposed by Lessard, Recht, and Packard in~\cite{lessard2014analysis} and is attractive because it requires solving a single SDP to obtain a bound that is valid for all subsequent iterations. This technique is particularly efficient for problems involving strong convexity, for which tight linear convergence rates are often recovered. However, as they aim at finding global rates of convergence, they are naturally more conservative than the general performance estimation approach.

For more details on the general topic of convergence analysis of first-order methods, we refer to the seminal books of Nemirovsky and Yudin~\cite{Book:NemirovskyYudin}, Polyak~\cite{Book:polyak1987}, Nesterov~\cite{Book:Nesterov}, and the more recent book of Bertsekas~\cite{bertsekas2015convex}. Concerning the development of accelerated methods, we specifically refer to the original work of Nesterov~\cite{Nesterov:1983wy,Book:Nesterov}, and to the later extensions to minimize smoothed convex functions~\cite{nesterov2005smooth} and composite functions~\cite{beck2009fast,nesterov2007gradient}.

\subsection{Paper organization and main contributions}
This work is divided into three main parts. First, \secref{sec:prox-PEP} is concerned with putting in place the performance estimation framework for large classes of first-order algorithms, objective functions, performance criteria, and initialization conditions. The main idea of this section is to require every element of the~\ref{Intro:PEP} to be \emph{linearly Gram-representable} \rev{(defined in Section~\ref{sec:tract_pep})}. This section contains multiple examples of standard settings for which the methodology applies---including those covering (sub)gradient methods (along with their projected and proximal counterparts) and conditional gradient methods (CGMs).

Section~\ref{sec:cvx_interp} focuses on providing convex interpolation conditions for different classes of convex functions commonly arising in practice. Those classes include convex functions, possibly with strong convexity, smoothness, bounded domain, and bounded (sub)gradient requirements.
The subclasses of indicator and support functions are also explicitly handled. Those classes of functions can all be used directly in the performance estimation framework of \secref{sec:prox-PEP}, since their corresponding interpolation conditions are {linearly Gram-representable}. This section ends with an extension of the convex interpolation results to cope with smooth nonconvex functions in a linearly Gram-representable way.

In \secref{sec:numerics}, we apply our approach to several concrete first-order algorithms. We obtain improvements on the analysis of several well-known methods, either analytically or numerically, including the proximal point algorithm and the CGM.
\del{We also use those results to provide an extension of the OGM proposed by Kim and Fessler~\cite{kim2014optimized} that incorporates a projection or a proximal operator to tackle constrained and composite problems.}

\section{Performance estimation framework for first-order algorithms}
\label{sec:prox-PEP} We start this section by formulating~\eqref{Intro:dPEP} in terms of a Gram matrix. This leads to a tractable convex formulation for~\eqref{Intro:dPEP}---once appropriate assumptions are made on the classes of objective function components, methods, performance criteria, and initialization conditions.  Those assumptions are motivated by practical applications, which we also provide in the following. The main point underlying those assumptions is to ensure that every element of the PEP can be formulated in a linear way in terms of both the entries of a Gram matrix and the function values at the iterates.
\subsection{Gram representations} \label{sec:GramRep} Let us consider $N+1$ iterates $x_0,\hdots,x_N$ and an optimal solution $x_*$, and the set of corresponding oracle outputs  $\left\{(f^{(k)}_i,g^{(k)}_i)\right\}_{i\in I,k\in K}$. The accumulated information after those $N+1$ oracle calls can be gathered into a $d \times (n+1)(N+2)$ matrix\footnote{We recall that $B:\E\rightarrow\Es$ is a positive definite operator which is chosen as the identity operator in standard situations (see Section~\ref{subsec:notations}).} $P_N$ (using a slight abuse of notation) and a vector $F_N$ of length $n(N+2)$:
\begin{align}
P_N&=[Bx_0 \ \hdots \ Bx_N \ \lvert \ Bx_* \ \ \lvert \ g_0^{(1)}  \ \hdots \ g_0^{(n)}\ \lvert \ \hdots \ \lvert \ g_N^{(1)}  \ \hdots \ g_N^{(n)} \ \lvert \ g_*^{(1)}  \ \hdots \ g_*^{(n)}],\label{eq:acc_info}\\
F_N&=[ \ f^{(1)}_0 \ \hdots \ f^{(n)}_0 \ \lvert \ \hdots \ \lvert \ f^{(1)}_N \  \hdots f^{(n)}_N \ \lvert \ f^{(1)}_*  \ \hdots f^{(n)}_*].\label{eq:acc_info_func}
\end{align}
We also denote by $B^{-1}P_N$ the matrix 
\[B^{-1}P_N=[x_0 \ \hdots \ x_N \ \lvert \ x_* \ \ \lvert \ B^{-1}g_0^{(1)}  \  \hdots \ B^{-1}g_*^{(n)}].\]
In order to formulate~\eqref{Intro:PEP} in a tractable way for first-order methods, we use a Gram matrix. That is, we define a symmetric $(n+1)(N+2) \times (n+1)(N+2)$ Gram matrix $G_N\in \mathbb{S}^{(n+1)(N+2)}$, using the following construction:
{\small \[ G_N = \begin{pmatrix}
\innerE{x_0}{x_0} & \hdots & \innerE{x_0}{x_N} & \innerE{x_0}{x_*} & \innerS{\rev{g_0^{(1)}}}{x_0} & \hdots & \innerS{g_*^{(n)}}{x_0}\\
\vdots & \ddots & \vdots & \vdots & \vdots & \ddots & \vdots \\
\innerE{x_N}{x_0} & \hdots & \innerE{x_N}{x_N} & \innerE{x_N}{x_*} & \innerS{\rev{g_0^{(1)}}}{x_N} & \hdots & \innerS{g_*^{(n)}}{x_N}\\
\innerE{x_*}{x_0} & \hdots & \innerE{x_*}{x_N} & \innerE{x_*}{x_*} & \innerS{\rev{g_0^{(1)}}}{x_*} & \hdots & \innerS{g_*^{(n)}}{x_*}\\
\innerS{\rev{g_0^{(1)}}}{x_0} & \hdots & \innerS{\rev{g_0^{(1)}}}{x_N} & \innerS{\rev{g_0^{(1)}}}{x_*} & \innerEs{\rev{g_0^{(1)}}}{\rev{g_0^{(1)}}} & \hdots & \innerEs{\rev{g_0^{(1)}}}{g_*^{(n)}}\\
\vdots & \ddots & \vdots & \vdots & \vdots & \ddots & \vdots \\
\innerS{g_*^{(n)}}{x_0} & \hdots & \innerS{g_*^{(n)}}{x_N} & \innerS{g_*^{(n)}}{x_*} & \innerEs{g_*^{(n)}}{\rev{g_0^{(1)}}} & \hdots & \innerEs{g_*^{(n)}}{g_*^{(n)}}
\end{pmatrix}\succeq 0.  \]}This can be written more compactly as $\left[G_N\right]_{ij}=\innerS{P_N e_i}{B^{-1}P_N e_j}=\innerEs{P_N e_i}{P_N e_j}$, where $P_N e_k$ corresponds to the $k$th column of $P_N$. Also, note that the size of this matrix does not depend on the dimension $d$ of the spaces we are working with.

\begin{remark} 
Note that Gram matrix  $G_N$ is positive semidefinite for any matrix $P_N$ (of the form~\eqref{eq:acc_info}). The number of linearly independent columns of $P_N$ is equal to the rank of $G_N$. Hence this rank is upper bounded by the dimension $d$ of the ambient space of the iterates. It is possible to recover a matrix $P_N$ of the form\footnote{In the case $\E=\Es=\R^d$ with the usual inner product $\inner{x}{y}=x^\top y$ and $B$ the identity operator, this can be done using the standard Cholesky factorization. In the general cases the exact same idea can be used, using the chosen inner product $\innerEs{.}{.}$ in the process.}~\eqref{eq:acc_info} from any Gram matrix $G_N\succeq 0$ satisfying $\mathrm{Rank}\ G_N\leq d$.
\label{remark:gram}
\end{remark}

Our goal for the next subsections is to show that in a lot of situations, the performance estimation problem~\eqref{Intro:dPEP} can be expressed exactly as an SDP in the $F_N$ and $G_N$ variables:
\modAT{\begin{align}\label{PEPS:SDP}
\sup_{
F_N\in\mathbb{R}^{n(N+2)}, G_N\in\mathbb{S}^{(n+1)(N+2)}
}  &c^\top\! F_N + \trace{C G_N} \tag{SDP-PEP} \\ \text{ subject to  } \quad
& a_i + b_i^\top F_N + \trace{D_i G_N} \leq 0 \quad \forall i \in S, \\
& G_N\succeq 0,
\end{align}}with $S$ some index set related to the constraints, and elements $a_i,b_i,c,D_i$, and $C$ of appropriate dimensions for writing the constraints and objective function linearly in terms of the Gram matrix $G_N$ and of the objective function values $F_N$.

\subsection{Tractable formulation of the performance estimation problem}\label{sec:tract_pep} \modAT{In this section, we present our main result, stating that computing the exact worst-case performance of a method on a class of functions is tractable and can, in many cases, be formulated as~\eqref{PEPS:SDP}. We start with the concept of Gram-representability for the different ingredients of the PEP.}
\begin{definition}A class of functions is Gram-representable (resp., linearly \\Gram-representable) if and only if its interpolation conditions~\eqref{def:cvx_interp_cond} can be formulated using a finite number of convex (resp., linear) constraints involving only the matrix $G_N$ and the function values $F_N$.
\label{def:GR_func}
\end{definition}

\modAT{
The functional classes of smooth strongly convex functions, smooth convex functions with bounded (sub)gradients, and strongly convex functions with bounded domain are linearly Gram-representable. In addition, the particular subclasses of support and indicator convex functions share this same advantageous property. The details and proofs of these results are postponed to \secref{sec:cvx_interp}.}

\begin{definition} A performance measure is Gram-representable (resp., linearly \\Gram-representable) if and only if it can be expressed as a concave (resp., linear) function involving only the matrix $G_N$ and the function values $F_N$.
\end{definition}

\modAT{The class of linearly Gram-representable performance criteria contains a large variety of choices, including most standard measures we are aware of. For example, it is easy to check that standard \rev{optimality criteria} in function values $F(x_N)-F(x_*)$, in residual subgradient norm $\normdsqsm{\tilde{\nabla}F(x_N)}$, distance to optimality $\normpsq{x_N-x_*}$, and distance to feasibility $\normpsq{x_N-\Pi_Q(x_N)}$, can be handled.

On the other hand, multiple examples of nonlinear Gram-representable performance criteria can also be handled with no difficulty. This includes performance measures involving the best values among all iterates, for example, $\min_{0 \le i \le N} F(x_i) - F(x_*)$, or the best residual gradient norm among the iterates $\min_{0 \le i \le N} \normdsq{\nabla F(x_i)}$ (see also \cite[Sect.\@ 4.3]{taylor2015smooth}).
}
\begin{definition} An initialization condition is Gram-representable (resp., linearly Gram-representable) if and only if it can be expressed using a finite number of convex (resp., linear) constraints involving only the matrix $G_N$ and the function values $F_N$.
\end{definition}

\modAT{Standard examples of valid initial conditions include the classical bounds on the initial distance to optimality $\normpsq{x_0-x_*}\leq R^2$, on the initial function value $F(x_0)-F_*\leq R$, and on initial gradient value $\normdsq{\nabla F(x_0)}\leq R^2$, \revd{for given values of $R\geq 0$}.}

\begin{definition} A first-order method is Gram-representable (resp., linearly \\Gram-representable) if and only if the computation of its iterates, implicitly defined by an equation of type \eqref{EqImplicite}, can be expressed using a finite number of convex (resp., linear) constraints involving only the matrix $G_N$ and the function values $F_N$.
\end{definition}

\modAT{We refer to the next section for examples of linearly Gram-representable methods.} \revd{Note that the necessary and sufficient condition for $x_*$ to be optimal for $F$ is always linearly Gram-representable. Indeed, it corresponds to requiring $\tilde\nabla F(x_*)=0$, i.e.,
\[ \sum_{k\in K} \tilde\nabla F^{(k)}(x_*) = \sum_{k\in K} g_*^{(k)} =0 \Leftrightarrow \normdsq{\sum_{k\in K} g_*^{(k)}}=\inneraEs{\sum_{k\in K} g_*^{(k)}}{\sum_{k\in K} g_*^{(k)}}=0, \]
where the last condition is linear in the entries of $G_N$.}

We can now state our main results concerning Gram-representable situations. 
\begin{proposition} Consider a class of composite objective functions 
 $\mathcal{F}_K(\E)$ with $n$ components, a first-order method $\mathcal{M}$, a performance measure $\mathcal{P}$, and an initial condition $\mathcal{I}$ which are all Gram-representable.

Computing the worst-case for criterion $\mathcal{P}$ of method $\mathcal{M}$ after $N$ iterations on objective functions in class $\mathcal{F}_K(\E)$ with initial condition $\mathcal{I}$ can be formulated as a convex \rev{program when} dimension of the space $\E$ satisfies $d\geq (n+1)(N+2)$. Otherwise, it can be formulated as a convex program plus an additional nonconvex rank constraint $\mathrm{Rank} \ G_N\leq d$. 

If in addition $\mathcal{F}_K(\E)$, $\mathcal{M}$, $\mathcal{P}$, and $\mathcal{I}$ are linearly Gram-representable, then the corresponding \revd{convex} problem is an SDP of the form~\eqref{PEPS:SDP}, with $F_N\in\mathbb{R}^{n(N+2)}$ and $G_N\in\mathbb{S}^{(n+1)(N+2)}$ as variables.
\label{thm:sdp_pep}
\end{proposition}
\begin{proof}
It directly follows from Remark~\ref{remark:gram} and from the definitions of (linear) Gram-representability for the class of functions, first-order methods, performance measures, optimality condition of a solution, and initialization conditions: any solution to the corresponding optimization problem can be transformed into a particular instance of~\eqref{eq:origOpt}, and vice versa.
\end{proof}
\begin{remark}The optimal value of~\eqref{Intro:PEP} increases with dimension $d$. \\When \eqref{Intro:PEP} with Gram-representable elements attains a finite optimal value, Proposition~\ref{thm:sdp_pep} implies the existence of a function with dimension \revv{at most} $(n+1)(N+2)$ that achieves \rev{the} worst-case value.\end{remark}

\begin{remark} \label{rem:largescale}The assumption $d\geq (n+1)(N+2)$ is referred to as the \emph{large-scale assumption} in what follows. In terms of PEP, this assumption allows us to discard the nonconvex rank constraint \revv{and lead to a tractable semidefinite programming problem, which can be solved to global optimality efficiently (see, e.g.,~\cite{Vandenberghe94semidefiniteprogramming}). Without that assumption, our PEP is a nonconvex rank-constrained SDP, equivalent to a quadratic programming problem that is NP-hard in general (e.g., it has MAX-CUT~\cite{goemans1995improved} and other nonconvex quadratic programs~\cite{pardalos1991quadratic,sahni1974computationally} as particular cases). Approaches to handle rank constraints exist (e.g., via augmented Lagrangian techniques~\cite{burer2003nonlinear}, via manifold optimization~\cite{journee2010low}, or via Newton-like methods~\cite{orsi2006newton}) but in general only guarantee convergence to stationary points. This is not useful in the case of~\eqref{PEPS:SDP}, as this only provides lower bounds on the worst-case performance.}
\end{remark}

\begin{remark}

Under the large-scale assumption, we obtain dimension-free guarantees (i.e., valid for any dimension, and tight as soon as $d \ge (n+1)(N+2)$), as is commonly found in the literature about first-order methods. In addition, we note that the dimension bound $(n+1)(N+2)$ is in fact (very) conservative for most standard algorithms---that is, the \revv{bound in the large-scale assumption can typically be significantly reduced; see Corollary~\ref{cor:Large_scale_FSLFOM}.
} 
\end{remark}

\begin{remark} The worst-case results provided by the SDP from \rev{Proposition}~\ref{thm:sdp_pep} provide a tight worst-case \modAT{achievable} for any operator $B$ and any dual pairing $\inner{.}{.}$. 
\label{rem:genB_wc}
\end{remark}

\subsection{Linearly Gram-representable first-order methods} \label{sec:lingramrep} This class of first-order methods contains as particular cases the class of FSLFOM, whose iterations are defined by a linear equation (with known constant coefficients) involving the iterates and the corresponding (sub)gradients.

\newpage\begin{definition}
An FSLFOM is a method which computes iterate $x_{i}$ as the solution of \footnote{The iteration is written as an \rev{equality} on $\E$, but it is possible and totally equivalent to write it on $\E^*$ using the operator $B^{-1}$.} 
\begin{align}
t_{i,i} Bx_{i}+\sum_{k\in K} h^{(k)}_{i,i} g^{(k)}_{i}=\sum_{j=0}^{i-1} \left[ t_{i,j} \rev{Bx_j}+ \sum_{k\in K} h_{i,j}^{(k)}g^{(k)}_j \right] ,\label{eq:gen_alg_model}\tag{FSLFOM}
\end{align}
where all coefficients $h_{i,j}^{(k)},t_{i,j}\in\mathbb{R}$ ($0\leq j\leq i$ and $k\in K$) are fixed beforehand.
\end{definition}

Note the class of FSLFOM is exactly the class of methods whose iterations can be written in the form (using first-order optimality conditions and convexity of $F^{(k)}$):
\begin{align*}
x_{i}=\argmin{x\in\E} \Bigg\{\frac{t_{i,i}}{2}\normpsq{x}&+ \sum_{k\in K} h^{(k)}_{i,i}F^{(k)}(x)\\-&\inner{\sum_{j=0}^{i-1} \left[ t_{i,j}B x_j+ \sum_{k\in K} h_{i,j}^{(k)}\nabla F^{(k)}(x_j)\right]}{x} \Bigg\} ,
\end{align*}
which, in some sense, describes the most general method our framework can deal with. The computation of iterate $x_i$ can also be written as the following linear equation, which involves a linear combination of the columns of matrix $P_N$ (which contain the harvested first-order information about the problem so far) using a constant vector of coefficients  $m_i\in\mathbb{R}^{(n+1)(N+2)}$:
\begin{align*}
P_N {m}_i=0.
\end{align*}
Note that coefficients in $m_i$ corresponding to columns describing subsequent iterates ($B x_j$ and $g^{(k)}_j$  $\forall j > i$ and $k \in K$)  must naturally be equal to zero, as well as those of columns related to the optimal solution ($B x_*$ and $g^{(k)}_*$  $\forall k \in K$). 
In addition, any FSLFOM is linearly Gram-representable using the following formulation:
\begin{equation}
0=P_N m_i \Leftrightarrow 0=\normdsq{P_N m_i}=\inner{P_N m_i}{B^{-1}P_N m_i}=m_i ^{\top\!}G_{N} m_i,\label{eq:trickFSLFOM}
\end{equation}
which is clearly linear in terms of the Gram matrix $G_N$. This can also easily be extended to cope with more general classes of linearly Gram-representable first-order methods, such as the following:
\begin{equation}
c_i^{(\text{low})\top} F_N +b_i^{(\text{low})} \leq m_i ^\top G_N m_i  \le c_i^{(\text{up})\top} F_N +b_i^{(\text{up})},
\label{eq:extd_FSLFOM}
\end{equation}
where $c_i^{(\text{low})},b_i^{(\text{low})}$ and $c_i^{(\text{up})},b_i^{(\text{up})}$ are some fixed parameters. Those could, for example, be used to require a sufficient decrease condition (involving function values in $F_N$) or to consider methods that perform an inexact computation of the next iterate in~\eqref{eq:gen_alg_model}, such as in 
\begin{equation}
\normd{P_N m_i} \le \epsilon_i \Leftrightarrow m_i^\top G_N m_i  \le \epsilon_i^2,
\label{eq:gen_alg_model_inex}\tag{Inexact FSLFOM}
\end{equation}
where $\epsilon_i\geq 0$ is some tolerance on the accuracy of the computation of $x_i$ in~\eqref{eq:gen_alg_model}. 

\medskip\paragraph{Examples of FSLFOM} Before going into the details of the PEPs for our class of linear fixed-step methods over different classes of convex functions, let us give several examples of methods fitting into the model provided by~\eqref{eq:gen_alg_model} and \eqref{eq:gen_alg_model_inex}. 

\begin{itemize}
\item \emph{Fixed-step subgradient and gradient algorithms}. Minimizing a convex function $F$ using a fixed-step subgradient method is naturally described as $x_{i}=x_{i-1}-\alpha_{i}B^{-1} g_{i-1}$, with $\alpha_{i}$ some step size and $g_{i-1}\in\partial F(x_{i-1})$. The method clearly belongs to the class of FSLFOM, and its linear Gram matrix representation can be obtained using formulation~\eqref{eq:trickFSLFOM}.
\item \emph{Proximal methods and proximal gradient methods}. Consider a composite objective $F^{(1)}+F^{(2)}$, where $F^{(2)}$ admits a computable proximal operator. Minimizing this objective with a fixed-step proximal gradient method is usually described as performing an explicit (sub)gradient step on $F^{(1)}$ followed by a (proximal) minimization step involving $F^{(2)}$:
\begin{align*}
x_{i}&=\prox{\alpha_i F^{(2)}}{x_{i-1}-\alpha_iB^{-1} \tilde{\nabla} F^{(1)}(x_{i-1})}\\&=\argmin{x\in\E}\left\{ \alpha_i F^{(2)}(x)+\frac{1}{2}\normpsq{x_{i-1}- \alpha_i B^{-1}\tilde{\nabla} F^{(1)}(x_{i-1})-x}\right\}.
\end{align*}
Optimality conditions on this last term allow writing each iteration as
\[B x_{i}+ \alpha_i \tilde{\nabla}F^{(2)}(x_{i})=B x_{i-1}- \alpha_i \tilde{\nabla} F^{(1)}(x_{i-1})\]
with some $\tilde{\nabla}F^{(2)}(x_{i})\in\partial F^{(2)}(x_{i})$, which is an implicit equation in $x_i$. This method is clearly a FSLFOM and therefore fits in our framework. Projected gradient methods are obtained using the same technique, but on the particular class of convex indicator functions $F^{(2)}$, whereas proximal point algorithms correspond to the case where $F^{(1)}=0$. 

\item \emph{Conditional gradient methods}. Consider an objective function $F^{(1)}$ to be minimized over a closed convex set $Q$, whose indicator function is  $F^{(2)}$. CGMs for this problem also fit the FSLFOM model. Indeed, their iterations take the following form (given a starting point $z_0$):
\begin{align*}
&y_i=\argmin{z\in\E}\left\{ \inner{z-z_i}{\tilde{\nabla} F^{(1)}(z_i)} +F^{(2)}(z)\right\},\\
&z_{i+1}=(1-\lambda_i)z_i + \lambda_i y_i,
\end{align*}
with coefficient $\lambda_i\in[0,1]$ chosen beforehand. Using first-order necessary and sufficient optimality conditions on the intermediate optimization problem, we obtain that $y_i$ can be defined by the following equation:
\begin{align*}
&\tilde{\nabla} F^{(1)}(z_i)=-\tilde{\nabla}F^{(2)}(y_{i}).
\end{align*}
This algorithm can also clearly be written as an FSLFOM; one only needs to merge the two sequences of iterates $y_i$ and $z_i$ into a single sequence, defining for example the iterates using  $x_{2i}=z_i$ and $x_{2i+1}=y_i$ for every $i=0,1,\hdots$

\item\emph{Inexact (sub)gradient methods}. Consider a convex function $F^{(1)}$ on which inexact steps are performed according to $x_{i+1}=x_i-\alpha_i B^{-1}(\tilde{\nabla} F^{(1)}(x_i)+\varepsilon_i)$, where errors $\varepsilon_i$ on the computation of the subgradients are bounded. More precisely, given some tolerance $\epsilon_i\geq 0$, we assume that $\normd{\varepsilon_i}\leq \epsilon_i$. This can be written in the inexact FSLFOM format:
\[ \normdsq{ \alpha_i^{-1}B \left(x_{i+1}-x_i\right)+\tilde{\nabla} F^{(1)}(x_i)}\le \epsilon_i^2. \]
Other noise models can also easily be used in the framework, such as the one proposed by d'Aspremont~\cite{d2008smooth}. However, the inexact $(\delta,L)$-oracles developed by Devolder, Glineur and Nesterov~\cite{devolder2014first} do not seem to easily fit into the approach.\footnote{This is due to the fact that no necessary and sufficient interpolation conditions for functions admitting such an inexact oracle are known---that is, standard conditions are only necessary to guarantee interpolability. Using necessary conditions that are not sufficient still allows obtaining upper bounds on the worst-case behavior, but those may not be tight.}
\end{itemize}
An even broader class of methods can be obtained by combining some of the above examples and/or restricting the functions to specific classes. For example, alternate projection-type algorithms are special cases of proximal methods applied to sums of convex indicator functions and hence can be represented in the FSLFOM format.

\subsection{\rev{Simplified} performance estimation problems}
{\rev{Note that for standard algorithms such as the above examples of FSLFOM, the SDP resulting from Proposition~\ref{thm:sdp_pep} can typically be further simplified, leading to a reduction in its size.}  

\rev{\begin{corollary} Consider a class of composite objective functions $\mathcal{F}_K(\E)$ with $n$ components, a performance measure $\mathcal{P}$ and an initialization condition $\mathcal{I}$ which are linearly Gram-representable, and an FSLFOM $\mathcal{M}$ whose iterations are linearly independent, meaning that the constant vectors $m_i$ used to define iterates $x_i$ in \eqref{eq:trickFSLFOM} are linearly independent.\footnote{This is a reasonable assumption, as every method using new information at each iteration will necessarily satisfy it. This does not imply that the points $x_i$ themselves are linearly independent.}

In addition, assume there are $p$ points $(g_i^{(k)},f_i^{(k)})$ such that neither $g_i^{(k)}$ nor $f_i^{(k)}$ is used in the performance measure $\mathcal{P}$, the initial condition $\mathcal{I}$, and the method $\mathcal{M}$. Then, the PEP can be written as a convex SDP using variables $F_N\in\mathbb{R}^{n(N+2)-p}$ and $G_N\in\mathbb{S}^{(n+1)(N+2)-N-p-1}$, with the possible additional rank constraint $\mathrm{rank}\ G_N \leq d$. 
\label{cor:Large_scale_FSLFOM}
\end{corollary}

\begin{proof}
One can remove from the original SDP formulation those $p$ unnecessary points corresponding to $p$ function values in variable $F_N$ and to $p$ rows/columns in the Gram matrix variable $G_N$. Furthermore, the $N$ equations defining the iterations allow us to further substitute $N$ variables, i.e., to remove $N$ columns from $P_N$ and hence $N$ rows/columns from the Gram matrix variable $G_N$. The dimension of $G_N$ can finally be decreased by one, using the fact \rev{that} one of the $g_*^{(k)}$ may also be discarded, by substituting it using the optimality condition defining $x_*$. 
\end{proof}

Under \rev{the assumptions of Corollary~\ref{cor:Large_scale_FSLFOM}}, the large-scale assumption becomes $d\geq (n+1)(N+2)-N-p \rev{-1}$. For example, when considering methods where \revv{only the output from a single oracle} (among the $n$ possible $F^{(k)}$) is used \revv{at each}  iteration, we have that $p=(n-1)(N+1)$, which leads to $d\geq N+n+2$. 

\revv{Furthermore,} for many standard performance measures such as objective function accuracy $F_N-F_*$ or distance to optimality $\normpsq{x_N-x_*}$, one arbitrary point $x_i$ may be fixed to zero because solutions to the SDP are invariant with respect to translations. This results then in the large-scale assumption $d\geq N+n+1$. For $n=1$, we recover the standard $d\geq N+2$ appearing in the case of a single component in the objective function~\cite{taylor2015smooth}.\label{rem:simp_SDP}

The original SDP from Proposition~\ref{thm:sdp_pep} may be challenging to solve in practice, because of its potentially large size on the one hand and because it may lack an interior on the other hand. We observe that the simplified PEP described above typically improves the situation for both issues, reducing the size of the problem and solving in a lot of cases the issue of a lack of interior \rev{points}.
}

\section{Convex interpolation}
\label{sec:cvx_interp}
\modAT{
In this section, we study convex interpolation problems for different standard classes of convex functions. The underlying motivation is to obtain discrete characterizations of convex functions commonly arising in the context of convex optimization via first-order methods. More specifically, the classes of convex functions of interest for this section are all linearly Gram-representable (see Definition~\ref{def:GR_func}). Therefore, using those classes within the performance estimation framework will lead to tractable formulations providing tightness guarantees.}

The main technical tools from this section are borrowed from convex analysis; we refer to the seminal works~\cite{bauschke2011convex,JBHU,Book:Rockafellar,rockafellar1998variational} for details.
\subsection{Functional characteristics}
\label{sec:func_char}
Consider a proper, closed and convex function $f$. The main characteristics of interest for us are the following, all commonly appearing in the context of first-order convex optimization:
\begin{itemize}
\item[(a)] \emph{Smoothness}: there exists some $L\in\mathbb{R}^{++}\cinf$ such that the inequality $\frac{1}{L}\normd{ g_1 - g_2}\leq \normp{x_1-x_2}$ holds for all pairs  $x_1, x_2\in\E$ and corresponding subgradients $g_1, g_2 \in \Es$ (i.e., such that $g_1\in\partial f(x_1)$ and $g_2\in\partial f(x_2)$).
\item[(b)] \emph{Strong convexity}: there exists some $\mu\in\mathbb{R}^+$ such that the function $f(x)-\frac{\mu}{2}\normpsq{x}$ is convex.
\item[(c)] \emph{Gradient boundedness}: there exists some $M\in\mathbb{R}^+\cinf$ such that $\normd{g}\leq M$ holds for all subgradients $g\in\Es$ (i.e., such that $\exists x:\ g\in\partial f(x)$).
\item[(d)] \emph{ Domain boundedness}: there exists some $ D\in\mathbb{R}^+\cinf$ such that $\normp{x}\leq D$ holds for all $x$ belonging to the domain $\left\{x\in\E: f(x)<\pinf\right\}$.
\end{itemize}
Alternatively, domain and gradient boundedness can be specified in terms of  diameters instead of radii. 
\begin{itemize}
\item[(c$^\prime$)] \emph{Gradient boundedness}: there exists some $M\in\mathbb{R}^+\cinf$ such that the inequality $\normd{g_1-g_2}\leq M$ holds for all subgradients $g_1,g_2\in\Es$ (i.e., such that $\exists x_1,x_2:\ g_1\in\partial f(x_1)$ and $g_2\in\partial f(x_2)$).
\item[(d$^\prime$)] \emph{Domain boundedness}: there exists $D\in\mathbb{R}^+\cinf$ such that $\normp{x_1-x_2}\leq D$ holds for all pairs $x_1, x_2$ belonging to the domain $\left\{x\in\E: f(x)<\pinf\right\}$.
\end{itemize}

As some characteristics are incompatible with each other (e.g., gradient boundedness is incompatible with strong convexity, domain boundedness is incompatible with smoothness), we define the following three classes of functions combining specific pairs of properties.
\begin{definition} Let $f:\fdef$ be a proper, closed, and convex function, denoted by $f\in\mathcal{F}_{0,\infty}$. We say that
\begin{itemize}
\item $f\in\mathcal{F}_{\mu,L}(\E)$ ($L$-smooth $\mu$-strongly convex functions) if it satisfies conditions \emph{(a)} and \emph{(b)} with $\mu<L$;
\item $f\in\mathcal{C}_{M,L}(\E)$ ($L$-smooth $M$-Lipschitz convex functions) if it satisfies conditions \emph{(a)} and \emph{(c)};  alternatively, $f\in\mathcal{C}'_{M,L}(\E)$ if it satisfies \emph{(a)} and \emph{(c$^\prime$)};
\item $f\in\mathcal{S}_{D,\mu}(\E)$ ($D$-bounded $\mu$-strongly convex functions) if it satisfies conditions \emph{(b)} and \emph{(d)}; alternatively $f\in\mathcal{S}'_{D,\mu}(\E)$ if it satisfies \emph{(b)} and \emph{(d$^\prime$)}.
\end{itemize}
\end{definition}

Note that boundedness and smoothness constants are allowed to take the value~$\pinf$, in order to allow the use of unbounded (domain or gradient) and nonsmooth functions as well. We handle those using  conventions $1/{\pinf}=0$ and $\pinf-c=\pinf$ for any $c\in\mathbb{R}$.
By assuming $\mu<L$, we exclude the classes $\mathcal{F}_{L,L}(\E)$ for $L\geq 0$. Those only contain quadratic functions of the form $f(x)=\frac{L}{2}\normpsq{x}+\inner{b}{x}+c$ for some $b\in\Es$ and $c\in\R$, for which it would be straightforward to obtain interpolation conditions.

A basic building block for the smooth convex interpolation conditions proposed in~\cite{taylor2015smooth} comes from Fenchel--Legendre conjugation. In particular, when considering functions $f$ in the class $\mathcal{F}_{0,\infty}(\E)$, the relationship $f\in\mathcal{F}_{\mu,L}(\E) \Leftrightarrow f^*\in\mathcal{F}_{1/L,1/\mu}(\Es)$ was intensively used to require smoothness of the convex interpolant. In the following, we additionally use for functions $f$ in $\mathcal{F}_{0,\infty}(\E)$ the duality correspondences $f\in\mathcal{C}_{M,L}(\E) \Leftrightarrow f^*\in\mathcal{S}_{M,1/L}(\Es)$ and its variant $f\in\mathcal{C}'_{M,L}(\E) \Leftrightarrow f^*\in\mathcal{S}'_{M,1/L}(\Es)$, in order to include boundedness properties in the convex interpolating functions, along with smoothness.
\begin{theorem} Consider a function $f\in\mathcal{F}_{0,\infty}(\E)$. We have $f\in\mathcal{C}_{M,\infty}(\E)$ (resp., $f\in\mathcal{C}'_{M,\infty}(\E)$) if and only if $f^*\in\mathcal{S}_{M,0}(\Es)$ (resp., $f^*\in\mathcal{S}'_{M,0}(\Es)$).
\label{thm:ConjCML_SRmu}
\end{theorem}
\begin{proof}
This follows from the equivalence: $g\in\partial f(x)\Leftrightarrow x\in\partial f^*(g) \Leftrightarrow f(x)+f^*(g)=\inner{g}{x}$ that holds for every function $f$ in $\mathcal{F}_{0,\pinf}(\E)$.
\end{proof}

\subsection{Interpolation conditions}
In this section, we provide interpolation conditions for the previously introduced three classes of functions. 
We start by recalling the following known interpolation result~\cite[Theorem 6]{taylor2015smooth}.\footnote{Theorem~\ref{thm:gencvxcomp} is formally proven in~\cite{taylor2015smooth} for the case of the standard inner product $\inner{x}{y}=x^\top\! y$ (and therefore also only for $\normstd{.}_2^2$). However, its proof can be rewritten in a completely straightforward manner to obtain the desired result for general inner products on $\E$ and self-adjoint positive definite linear operators $B$, and the corresponding induced primal and dual Euclidean norms.}
\begin{theorem} The set $\left\{(x_i,g_i,f_i)\right\}_{i\in I}$ is $\mathcal{F}_{\mu,L}$-interpolable if and only if the following set of conditions holds for every pair of indices $i \in I$ and $j \in I$:
\begin{align*}
f_i - f_j - \inner{g_j}{ x_i-x_j} \geq \frac{1}{2(1-\mu/L)}\bigg( \frac{1}{L}\normdsq{g_i-g_j} + &\mu \normpsq{x_i-x_j} \\ & - 2\frac{\mu}{L} \inner{g_j-g_i}{x_j-x_i}\bigg).
\end{align*}
\label{thm:gencvxcomp}
\end{theorem}

In particular, the simpler interpolation conditions for closed, convex proper functions (i.e., $\mathcal{F}_{0,\pinf}(\E)$ interpolation) are \begin{equation}
f_i-f_j-\inner{g_j}{x_i-x_j}\geq 0 \ \forall i,j \in I,\label{eq:ns_cvx_interp}
\end{equation}
which will serve to develop our next interpolation conditions. We start with $\mathcal{S}_{D,\mu}(\E)$-interpolability and later obtain $\mathcal{C}_{M,L}(\E)$-interpolation conditions using conjugation.

\begin{theorem} The set $\left\{(x_i,g_i,f_i)\right\}_{i\in I}$ is $\mathcal{S}_{D,\mu}$- ($D$-bounded, $\mu$-strongly convex) (resp., $\mathcal{S}'_{D,\mu}$-) interpolable if and only if the following set of conditions holds for every pair of indices $i \in I$ and $j \in I$:
\begin{align*}
&f_i - f_j - \inner{g_j}{x_i-x_j} \geq \frac{\mu}{2}\normpsq{x_i-x_j},\\
&\normp{x_j}\leq D \quad \text{(resp., } \normp{x_j-x_i}\leq D\text{)}.
\end{align*}
\label{thm:gencvx_SRmu}
\end{theorem}
\begin{proof}
Every function $f\in\mathcal{S}_{D,\mu}(\E)$ (resp., $f\in\mathcal{S}'_{D,\mu}(\E)$) satisfies the conditions. To prove that they are sufficient, consider the following construction:
\begin{align*}
f(x)=\left\{\begin{array}{ll}
\max_{i \in I} \left\{f_i+\inner{g_i}{x-x_i}+\frac{\mu}{2}\normpsq{x-x_i}\right\} \quad &  \text{if } x\in \text{conv}\left(\left\{x_i\right\}_{i\in I}\right), \\
\pinf & \text{elsewhere.}
\end{array}\right.
\end{align*}
Observe that $f$ is $\mu$-strongly convex (convex domain, and maximum of $\mu$-strongly convex functions) and that it does interpolate the set $\left\{(x_i,g_i,f_i)\right\}_{i\in I}$. First, we have 
\begin{align*}
f(x_j)&=\max_{i \in I} \left\{f_i+\inner{g_i}{x_j-x_i}+\frac{\mu}{2}\normpsq{x_j-x_i}\right\}\\
&\leq f_j,
\end{align*}
using interpolation conditions. By noting that the maximum is bigger than taking individually the component $j$, we also have that \[ \max_{i \in I} \left\{f_i+\inner{g_i}{x_j-x_i}+\frac{\mu}{2}\normpsq{x_j-x_i}\right\}\geq f_j,\]
which allows us to conclude that $f(x_j)=f_j$. To obtain that $g_j\in\partial f(x_j)$, let us write 
\begin{align*}
f(x)&=\max_{i \in I}\left\{f_i+\inner{g_i}{x-x_i}+\frac{\mu}{2}\normpsq{x-x_i}\right\}\\
&\geq \max_{i \in I} \left\{f_i+\inner{g_i}{x-x_i}\right\}\\
&\geq \rev{f_j+\inner{g_j}{x-x_j}}.
\end{align*}

Finally, note that $\text{conv}\left(\left\{x_i\right\}_{i\in I}\right)\subseteq B_{\E}(0,D)$, where $B_{\E}(0,D)$ is the ball centered at the origin with radius $D$ according to norm $\normp{.}$. Indeed, choosing $z=\sum_{i\in I}\lambda_i x_i$ with $\lambda_i\geq 0$ and $\sum_{i\in I}\lambda_i=1$, we have $\normp{z}\leq \sum_{i\in I}\lambda_i  \normp{x_i}\leq D,$ and $f$ has a bounded domain of radius $D$. Hence $\left\{(x_i,g_i,f_i)\right\}_{i\in I}$ is $\mathcal{S}_{D,\mu}$-interpolable, which concludes the proof for the $\mathcal{S}_{D,\mu}$ part.

To obtain the same result for $\mathcal{S}'_{D,\mu}$, note that $\forall y,z\in \text{conv}(\left\{x_i\right\}_{i \in I})$, we can write
$y=\sum_i \lambda_i x_i$ and $z=\sum_i \gamma_i x_i$ with $\lambda_i,\gamma_i\geq0$ and $\sum_i \lambda_i=\sum_i \gamma_i=1$. Hence,
\modAT{$
\normp{y-z}\leq \sum_i \lambda_i \sum_j \gamma_j \normp{ x_i -  x_j} \leq D.$}
\end{proof}

This interpolation result can be used immediately to develop interpolation conditions for the class of convex functions with bounded gradient, using the conjugate duality between smoothness and strong convexity on the one hand and gradient and domain boundedness on the other hand.
\begin{theorem} The set $\left\{(x_i,g_i,f_i)\right\}_{i\in I}$ is $\mathcal{C}_{M,L}$- ($L$-smooth with $M$-bounded subgradients) (resp., $\mathcal{C}'_{M,L}$-) interpolable if and only if the following set of conditions holds for every pair of indices $i \in I$ and $j \in I$:
\begin{align}
&f_i - f_j - \inner{g_j}{x_i-x_j} \geq \frac{1}{2L}\normdsq{g_i-g_j},\label{cond:CML1}\\
&\normd{g_j}\leq M \quad \text{(resp., } \normd{g_j-g_i}\leq M\text{)}.\label{cond:CML2}
\end{align}
\label{thm:gencvx_CML}
\end{theorem}
\begin{proof}
Note that a function $f\in\mathcal{C}_{M,L}(\E)$ (resp., $f\in\mathcal{C}'_{M,L}(\E)$) interpolates the set $\left\{(x_i,g_i,f_i)\right\}_{i\in I}$ if and only if there exists a corresponding conjugate function\break $f^*\in\mathcal{S}_{M,1/L}(\Es)$ (resp., $f^*\in\mathcal{S}'_{M,1/L}(\Es)$) interpolating the conjugate set\break $\{(g_i,x_i,\inner{g_i}{x_i}-f_i)\}_{i\in I}=\{(\tilde{x}_i,\tilde{g}_i,\tilde{f}_i)\}_{i\in I}$ (see \secref{sec:func_char}). Using interpolation conditions from Theorem~\ref{thm:gencvx_SRmu}, such a conjugate function $f^*$ exists if and only if
\begin{align*}
&\tilde{f}_i-\tilde{f}_j-\inner{\tilde{x}_i-\tilde{x}_j}{\tilde{g}_j}\geq \frac{1}{2L} \normdsq{\tilde{x}_i-\tilde{x}_j},\\
&\normd{\tilde{x}_j}\leq M  \quad \text{(resp., } \normd{\tilde{x}_j-\tilde{x}_i}\leq M\text{)},
\end{align*}
which are respectively equivalent to conditions~\eqref{cond:CML1} and~\eqref{cond:CML2}.
\end{proof}

\subsection{Indicator and support functions}
The use of projection (to deal with constraints) and regularization is so recurrent in optimization that we dedicate the next lines to interpolation procedures specifically tailored to deal with them.  

\paragraph{Indicator functions} 
In our setting, an indicator function is a closed convex function taking only values $0$ and $\infty$, for which it can be shown that the domain must be a closed convex set. As explained earlier, this class of \rev{functions} is particularly interesting when considering projection operators in the context of performance estimation, as a proximal step over an indicator function is equivalent to a projection on its domain.

Given such a proper and closed convex function $i:\E\rightarrow\left\{0,\pinf\right\}$, we say that it is a $D$-bounded indicator function (which we denote by $f\in\mathcal{I}_{D}(\E)$---resp., $f\in\mathcal{I}'_{D}(\E)$) if there exists a radius (resp., a diameter) $0\leq D \leq \pinf$ such that $\normp{x}\leq D$ (resp., $\normp{x_1-x_2}\leq D$) holds for all $x$ belonging to the domain $\left\{x: i(x)=0\right\}$ (resp., for all $x_1, x_2$ belonging to the domain $\left\{x: i(x)=0\right\}$).

This corresponds to a particular case of the $\mathcal{S}_{D,\mu}$- (or $\mathcal{S}'_{D,\mu}$-) interpolation problem with $\mu=0$. Note, however, that indicator function interpolation is not completely straightforward from $\mathcal{S}'_{D,\mu}$-interpolation, as, for example, requiring the corresponding interpolation constraints in addition to $f_i=0$ would not a priori guarantee that the interpolated function from Theorem~\ref{thm:gencvx_SRmu} would satisfy $f(x)=0$ on $\dom f$. 
\begin{theorem}
 The set $\left\{(x_i,g_i,f_i)\right\}_{i\in I}$ is $\mathcal{I}_D$- (resp., $\mathcal{I}'_D$-) interpolable, i.e., interpolable by a $D$-bounded indicator, if and only if the following inequalities hold  for every pair of indices $i \in I$ and $j \in I$:
\begin{align}
&f_i=0,\notag\\
&\inner{g_j}{x_i-x_j}\leq 0, \label{Eq:Id-Interp}\\
&\normp{x_i}\leq D  \quad \text{(resp., } \normp{x_j-x_i}\leq D\text{)}.\notag
\end{align}
\label{thm:iffproj}
\end{theorem}
\begin{proof} Any function $f\in\mathcal{I}_D(\E)$ (resp., $f\in\mathcal{I}'_D(\E)$) satisfies those conditions. To prove that they are sufficient, let us construct a convex set whose indicator  function interpolates the set $\left\{(x_i,g_i,0)\right\}$. That is, we construct a closed convex set $Q$ containing all $x_i$'s, for which $\normp{x}\leq D$ holds $\forall x\in Q$ (resp., $\normp{x-y}\leq D$ $\forall x,y \in Q$), and such that $\inner{g_i}{x-x_i}\leq 0$ holds $\forall x\in Q$. 

We start with the simpler case $D=\pinf$, by considering the polyhedral set $$Q=\left\{x\in\E \ | \ \inner{a_j}{x}\leq b_j \ \forall j\in I\right\}$$
with $a_j=g_j$ and $b_j=\inner{g_j}{x_j}$. The construction guarantees that $x_i\in Q$. Indeed, by condition~\eqref{Eq:Id-Interp} we have $\inner{g_j}{x_i}\leq \inner{g_j}{x_j},$
which is equivalent to $\inner{a_j}{x_i}\leq b_j$ using the definitions of $a_j$ and $b_j$, and therefore guarantees that $x_i\in Q$. 

In order to add the boundedness requirement, we replace the set $Q$ by the following $\tilde{Q}=Q\,\cap \text{ conv}(\left\{x_i\right\}_{i\in I})$. 
This new set $\tilde{Q}$ is still convex (intersection of two convex sets); it also trivially still satisfies inclusions $x_i\in \tilde{Q}$ (which are by construction both contained in $Q$ and $\text{conv}(\left\{x_i\right\}_{i})$) and conditions $\inner{g_i}{x-x_i}\leq 0$ $\forall x\in \tilde{Q}$ (since $\tilde{Q}\subseteq Q$). In addition, \rev{$\tilde{Q}$} has a radius bounded above by $D$, because $D$ is an upper bound on the radius (resp., diameter) of $\text{conv}(\left\{x_i\right\}_{i\in I})$. It is therefore clear that the indicator function $I_{\tilde{Q}}\in\mathcal{I}_D(\E)$ (resp., $\mathcal{I}'_D(\E)$) interpolates $\left\{(x_i,g_i,0)\right\}_{i\in I}$.
\end{proof}

\paragraph{Support functions}  It is a standard observation that support functions are convex conjugates of indicator functions. Indeed, the support function for the closed convex set $Q\subseteq \E$ is defined as
\[\sigma_Q(s)=\sup_{x\in Q} \inner{s}{x}=\sup_{x\in\E} \inner{s}{x}-I_Q(x).\]
Support functions are very commonly used in applications. In particular, all norms, which are used for regularization, are support functions (e.g., the $l_1$ norm is the support function of the unit ball for $\normstd{.}_\infty$). 

Denote the set of support functions with an $M$-Lipschitz condition by $\mathcal{I}^*_M(\E)$ (resp., $\mathcal{I}'^*_M(\E)$). Using conjugacy, interpolation conditions for indicator functions immediately give us the equivalent result for support functions.  Indeed, requiring a set $S=\left\{(x_i,g_i,f_i)\right\}_{i\in I}$ to be $\mathcal{I}^*_M$- (resp., $\mathcal{I}'^*_M$-) interpolable is equivalent to requiring the set $\tilde{S}=\left\{(g_i,x_i,\inner{g_i}{x_i}-f_i)\right\}_{i\in I}$ to be $\mathcal{I}_M$- (resp., $\mathcal{I}'_M$-) interpolable, and we obtain the following consequence of Theorem~\ref{thm:iffproj}.
\begin{corollary}  The set $\left\{(x_i,g_i,f_i)\right\}_{i\in I}$ is $\mathcal{I}^*_M$- (resp., $\mathcal{I}'^*_M$-) interpolable, i.e., interpolable by a support function with $M$-bounded subgradients, if and only if the following inequalities hold for every pair of indices $i \in I$ and $j \in I$:
\begin{align*}
\inner{g_i}{x_i}-f_i&=0,\\
\inner{g_i-g_j}{x_j}&\leq 0,\\
\normd{g_i}&\leq M \quad \text{(resp., }\normd{g_i-g_j}\leq M\text{)}.
\end{align*}
\label{cor:iffsupp}
\end{corollary}
\subsection{Smooth nonconvex interpolation} \label{sec:smoothnconvex}In this short section, we derive interpolation conditions for smooth, not necessarily convex, functions. Those conditions are also linearly Gram-representable and can be used to obtain tight versions of~\eqref{Intro:dPEP} for nonconvex optimization. 
\begin{definition} Let $L\in\R^{+}$. A differentiable function $f:\fdef$ is $L$-smooth, denoted by $f\in\mathcal{F}_{-L,L}(\E)$), if it satisfies the following condition $\forall x,y\in\E$:
\[ \big\lvert f(x)+\inner{\nabla f(x)}{y-x}-f(y)\bigr\rvert \leq \frac{L}{2}\normpsq{x-y}.\]
\end{definition}

The following lemma will be used to derive interpolation conditions for $f\in\mathcal{F}_{-L,L}(\E)$ from the smooth convex case.
\begin{lemma} Let $L\in\R^{+}$, and consider a function $f:\fdef$. We have the equivalence $f\in\mathcal{F}_{-L,L}(\E)\Leftrightarrow f+\frac{L}{2}\normpsq{x}\in\mathcal{F}_{0,2L}(\E)$.\label{lem:fornoncvxsmoothinterp}
\end{lemma}
\begin{proof} Let $f:\fdef$ and define $h(x)=f(x)+\frac{L}{2}\normpsq{x}$. Since we have that $\nabla h(x)=\nabla f(x)+LBx$, it follows that, $\forall x,y\in\E$,
\begin{align*}
f(x)+\inner{\nabla f(x)}{y-x}-f(y) \leq \frac{L}{2}\normpsq{x-y}\Leftrightarrow h(y)\geq h(x)&+\inner{\nabla h(x)}{y-x},\\
-f(x)-\inner{\nabla f(x)}{y-x}+f(y)\leq \frac{L}{2}\normpsq{x-y}\Leftrightarrow 
h(y)\leq h(x)&+\inner{\nabla h(x)}{y-x}\\&+L\normpsq{x-y},
\end{align*}
where the equivalences are obtained by expressing $f$ and $\nabla f$ in terms of $h$ and $\nabla h$ (or reciprocally), which proves our statement.
\end{proof}

From Lemma~\ref{lem:fornoncvxsmoothinterp} and Theorem~\ref{thm:gencvxcomp}, it is now straightforward to establish the desired interpolation conditions.
\begin{theorem}  Let $L\in\R^{++}$, the set $\left\{(x_i,g_i,f_i)\right\}_{i\in I}$ is $\mathcal{F}_{-L,L}$- ($L$-smooth-) interpolable if and only if the following inequality holds $\forall i,j\in I$:
\begin{align*}
f_i\geq f_j-\frac{L}{4}\normpsq{x_i-x_j}{\color{red}\text{\boldmath$-$}}\frac{1}{2}\inner{g_i+g_j}{x_j-x_i}+\frac{1}{4L}\normdsq{g_i-g_j}.
\end{align*}
\label{thm:noncvxinterp}
\end{theorem}
\begin{proof}
As $L$ is positive and finite, the statement follows from the equivalence\break between $\mathcal{F}_{-L,L}$-interpolability of the set $\left\{\left(x_i,g_i,f_i\right)\right\}_{i\in I}$ and $\mathcal{F}_{0,2L}$-interpolability of the set $\left\{\right(x_i,g_i+LBx_i,f_i+\frac{L}{2}\normpsq{x_i}\left)\right\}_{i\in I}$.
\end{proof}
\section{Algorithm analysis} 
\label{sec:numerics}
In this section, we analytically and numerically study different algorithms for solving variants of~\eqref{eq:origOpt} and compare our results with standard guarantees from the literature.\footnote{Note that most of the literature results are presented when $B$ is the identity operator (and hence $\E=\Es$). We will nevertheless compare our slightly more general results with the standard bounds from the literature (thus even when they are officially valid only for $B$ being the identity)---we recall that our results are valid for any self-adjoint positive definite linear operator $B:\E\rightarrow\Es$ (see Remark~\ref{rem:genB_wc}).} \modAT{We begin with an analytical study of a proximal point algorithm (\secref{ssec:ppa}). This is followed by a comparison between several standard variants of fast proximal gradient methods (\secref{ssec:FPGMs}) using the PEP approach. \del{On the way, we propose an extension of the optimized gradient method (OGM) proposed by Kim and Fessler~\cite{kim2014optimized}.} Finally, we conclude by applying our framework to a conditional gradient method (\secref{ssec:FW_cgm}) and to two alternate projections schemes (\secref{ssec:APM_DAPM}). Those choices illustrate the applicability of the approach for studying a large variety of methods and performance measures.}

\subsection{A proximal point algorithm}\label{ssec:ppa}

Consider a simple model with only one convex (possibly nonsmooth) term in the objective function,
\[\min_{x\in\E} F(x),\]
with $F\in\mathcal{F}_{0,\infty}(\E)$. In this first example, we assume that the following proximal operation is easy to compute for $F$ and defines the next iterate (using a given step size $\alpha_{k+1}$):
\begin{align*}
x_{k+1}&=\prox{\alpha_{k+1} F}{x_{k}}=\argmin{x\in\E}\left\{ \alpha_{k+1} F(x)+\frac{1}{2}\normpsq{x_k-x}\right\}.
\end{align*}
Using an observation made in Section~\ref{sec:lingramrep}, we see that iterations can also be written in the form of an implicit method  $x_{k+1}=x_k-\alpha_{k+1}B^{-1}g_{k+1}$, for some $g_{k+1}\in\partial F(x_{k+1})$, and hence belong to the class~\eqref{eq:gen_alg_model}.

For a recent overview and motivations concerning proximal algorithms, we refer the reader to the work of Combettes and Pesquet\footnote{This work among others features a large list of known proximal operators.}~\cite{combettes2011proximal} and to the review works of Bertsekas~\cite{bertsekas2011incremental} and Parikh and Boyd~\cite{parikh2013proximal}. For a historical point of view on those methods, we refer to the pioneer works of Moreau~\cite{moreau1965proximite} and Rockafellar~\cite{rockafellar1976monotone} and the analysis of G\"uler~\cite{guler1991convergence}.

{\small
\begin{center}
\fbox{
\parbox{0.9\textwidth}{
        \textbf{Proximal point algorithm (PPA)}
 \begin{itemize}
  \item[] Input: $F\in\mathcal{F}_{0,\infty}(\E)$, $x_0\in\E$. Parameters: $\left\{\alpha_k\right\}_{k\ge 1}$ with $\alpha_k > 0$.\\[-0.2cm]
  \item[] For $k=1:N$\\[-0.5cm]
      \begin{align*}
    &x_{k}=\prox{\alpha_k F}{x_{k-1}}
    \end{align*}
  \end{itemize}
   }}
   \end{center}
   }
\subsubsection{Convergence of PPA in function and gradient values} The standard convergence result for the proximal point algorithm is provided by G\"uler in~\cite[Theorem 2.1]{guler1991convergence}:
$$F(x_N)-F_*\leq \frac{R^2}{2\sum_{k=1}^N \alpha_k}$$
for any initial condition $x_0$ satisfying $\normp{x_0-x_*}\leq R$. We are able to divide this bound by $2$ using the PEP approach. 

\begin{theorem}
Let $\left\{\alpha_k\right\}_k$ be a sequence of positive step sizes and $x_0$ some initial iterate satisfying $\normp{x_0-x_*}\leq R$ for some optimal point $x_*$. Any sequence $\left\{x_k\right\}_k$ generated by the proximal point algorithm with step sizes $\left\{\alpha_k\right\}_k$ applied to a function $F\in\mathcal{F}_{0,\infty}(\E)$ satisfies \[F(x_N)-F_*\leq \frac{R^2}{4\sum_{k=1}^N \alpha_k} \]
and this bound cannot be improved, even in dimension one ($\dim \E=\dim \Es=1$).
\label{thm:PPA_conv}
\end{theorem}
\begin{proof}
We first prove that the bound is tight. For given $N$, $R$ and step sizes $\{ \alpha_k \}_{1 \le k \le N}$, we consider the $l_1$-shaped one-dimensional function \[ F(x)=\frac{\sqrt{B}R\lvert x\rvert}{2\sum_{k=1}^N \alpha_k}=\frac{R\normp{x}}{2\sum_{k=1}^N \alpha_k}\;,\] for which $x_* = 0$ and $F_*=0$. Applying $N$ iterations of PPA with step sizes $\{ \alpha_k \}_{1 \le k \le N}$ to this one-dimensional function, starting from $x_0=-\frac{R}{\sqrt{B}}$ (which satisfies $\normp{x_0-x_*}\leq R$), leads to a sequence whose last iterate satisfies \[ F(x_N)-F_*=\frac{R^2}{4\sum_{k=1}^N\alpha_k} \;. \] 
\noindent Indeed, note that for $x\neq 0$, we have $\nabla F(x)=\textrm{sign}(x) \frac{\sqrt{B}R}{2\sum_{k=1}^N\alpha_k}$. Hence, \[x_N=x_0+B^{-1}\sum_{k=1}^N \alpha_k \frac{\sqrt{B}R}{2\sum_{k=1}^N\alpha_k}=-\frac{R}{2\sqrt{B}} \] which implies the desired result.

\modAT{The proof of the upper bound is based on considering a simplified formulation of~\eqref{Intro:dPEP} for the proximal point algorithm, computing its dual and exhibiting a feasible solution to that dual. Because it is a little longer it is relegated to Appendix~\ref{sec:app_proof_thm_ppa1}.}
\end{proof}

\modAT{Let us consider another convergence measure based on the residual subgradient norm. Studying a PEP similar to the one above, we obtained strong numerical evidence for the following conjecture.}
\begin{conjecture}
Let $\left\{\alpha_k\right\}_k$ be a sequence of positive step sizes and $x_0$ some initial iterate satisfying $\normp{x_0-x_*}\leq R$ for some optimal point $x_*$. For any sequence $\left\{x_k\right\}_k$ generated by the proximal point algorithm with step sizes $\left\{\alpha_k\right\}_k$ on a function $F\in\mathcal{F}_{0,\infty}(\E)$, there exists for every iterate $x_N$ \revv{a subgradient} $\rev{g_N}\in\partial F(x_N)$  \revv{such that} \rev{\[ \normd{g_N}\leq \frac{R}{\sum_{k=1}^N \alpha_k}.\]}In particular, the choice \rev{$g_N=\frac{Bx_{N-1}-Bx_N}{\alpha_N}$ \revv{is a subgradient satisfying the inequality}}. 
\label{conj:PPA_conv_grad}
\end{conjecture}

\revv{Observe that this bound cannot be improved, as it is attained on the (one-dimensional) $l_1$-shaped function $F(x)=\frac{\sqrt{B}R\lvert x\rvert}{\sum_{k=1}^N \alpha_k}$ started from  $x_0=-R/\sqrt{B}$. The particular choice of subgradient suggested in the theorem corresponds to the subgradient appearing in the proximal operation when written as an implicit subgradient step.}

 This sort of convergence results in terms of the residual (sub)gradient norm is particularly interesting when considering dual methods. In that case, the dual residual gradient norm corresponds to the primal distance to feasibility (see, e.g.,~\cite{devolder2012double}).
 
\subsection{Fast gradient methods}\label{ssec:FPGMs}

In this section, we consider the two-term composite objective function
\begin{equation}
\min_{x\in\E} \left\{F(x)\equiv F^{(1)}(x)+F^{(2)}(x)\right\}\label{eq:comp_regsmooth}
\end{equation} 
with $F^{(1)}\in\mathcal{F}_{0,L}(\E)$ (smooth convex function) and $F^{(2)}\in\mathcal{F}_{0,\infty}(\E)$ (nonsmooth convex function). We assume that gradients are easy to compute for $F^{(1)}$ and that the proximal operation is easy to compute for $F^{(2)}$:
\begin{align*}
\prox{\alpha F^{(2)}}{x}=\argmin{y\in\E}\left\{ \alpha F^{(2)}(y)+\frac{1}{2}\normpsq{x-y}\right\}.
\end{align*}
In order to approximatively solve~\eqref{eq:comp_regsmooth}, it is common to use different variants of fast proximal gradient methods (FPGM). We numerically investigate the worst-case guarantees of two variants using different step size policies and propose new variants with slightly better worst-case behaviors. Also, we highlight differences in the worst-case performances obtained in the cases where $F^{(2)}=0$ (unconstrained smooth convex minimization), $F^{(2)}\in\mathcal{I}_{\pinf}(\E)$ (constrained smooth convex minimization), and the general $F^{(2)}\in\mathcal{F}_{0,\pinf}(\E)$ (nonsmooth composite convex minimization).

\rev{In the following, we call the standard fast proximal gradient method FPGM1 (FISTA~\cite{beck2009fast}) and introduce FPGM2, a variant with slightly better guarantees, and POGM, a novel proximal version of the optimized gradient method~\cite{kim2014optimized}. FPGM2 and POGM illustrate how PEPs can be used in the development of new optimization algorithms; their study in this paper remains, however, entirely numerical.}

\subsubsection{Standard fast proximal gradient methods (FPMG1)}  The first variants of accelerated proximal methods we are considering use a standard proximal step after an explicit gradient step for generating the so-called {primary sequence}~$\left\{y_k\right\}_k$.
\vspace{.05cm}
\begin{center}
{\small \fbox{
\parbox{0.9\textwidth}{
        \textbf{Fast proximal gradient method (FPGM1)}
  \begin{itemize}
  \item[] Input: $F^{(1)}\in\mathcal{F}_{0,L}(\E)$, $F^{(2)}\in\mathcal{F}_{0,\infty}(\E)$ $x_0\in\E$, $y_0=x_0$.\\[-0.2cm]
  \item[] For $k=1:N$\\[-0.5cm]
      \begin{align*}
    &y_{k}=\prox{F^{(2)}/L}{x_{k-1}-\frac{1}{L}B^{-1}\nabla F^{(1)}(x_{k-1})}\\
    &x_{k}=y_{k}+\alpha_{k} (y_{k}-y_{k-1})
    \end{align*}
  \end{itemize}
       }}}
\end{center}
\vspace{.3cm}

In this algorithm, we refer to coefficients $\alpha_k$ as \emph{inertial parameters}. We use two standard variants: $\alpha^{(a)}_{k}=\frac{k-1}{k+2}$---among others proposed in~\cite{su2014differential,tseng2008accelerated}---and $\alpha^{(b)}_{k}=\frac{\theta_{k-1}-1}{\theta_{k}}$, with \[\theta_{k}=\frac{1+\sqrt{4\theta_{k-1}^2+1}}{2}\] and $\theta_0=1$--- see~\cite{beck2009fast,Nesterov:1983wy,tseng2008accelerated}. For both variants, the standard convergence result is (see, e.g.,~\cite{beck2009fast,su2014differential})
\begin{align}
F(y_N)-F_*\leq \frac{\rev{2LR^2}}{(N+1)^2}\label{eq:th_FGM}
\end{align}
\rev{for any initial iterate $x_0$ such that $\normp{x_0-x_*}\leq R$.}
We numerically compare those two variants of FPGM1 using~\eqref{Intro:dPEP} in \figref{Fig:FGMS_comp} (left plot). After $100$ iterations, both inertial parameter policies behave about the same way (parameters $\alpha^{(b)}_k$ perform only about $2\%$ better than $\alpha^{(a)}_k$ in terms of worst-case performances). We also observe that the behavior of both variants of FPGM1 is well captured by the standard guarantee~\eqref{eq:th_FGM}.

\begin{figure}[!ht]
\begin{center}
\subfigure{\includegraphics[scale=0.3]{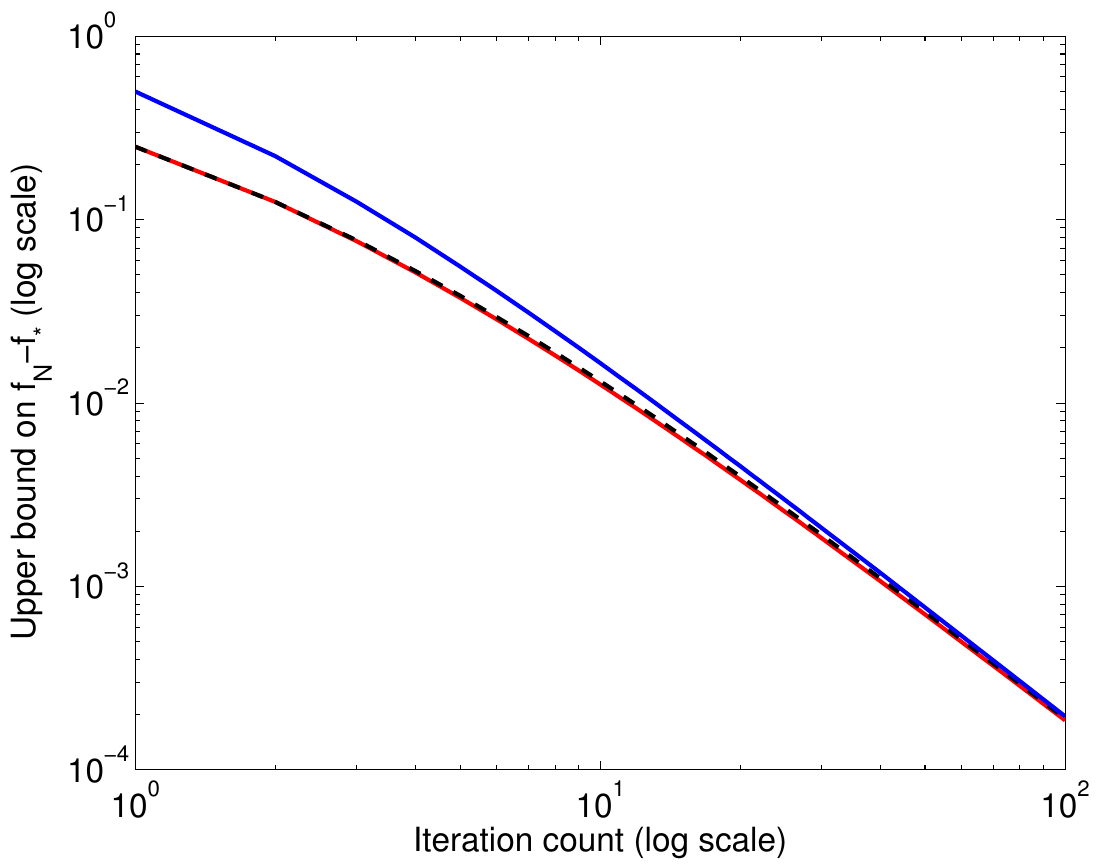}}
\hspace{.1cm}
\subfigure{\includegraphics[scale=0.3]{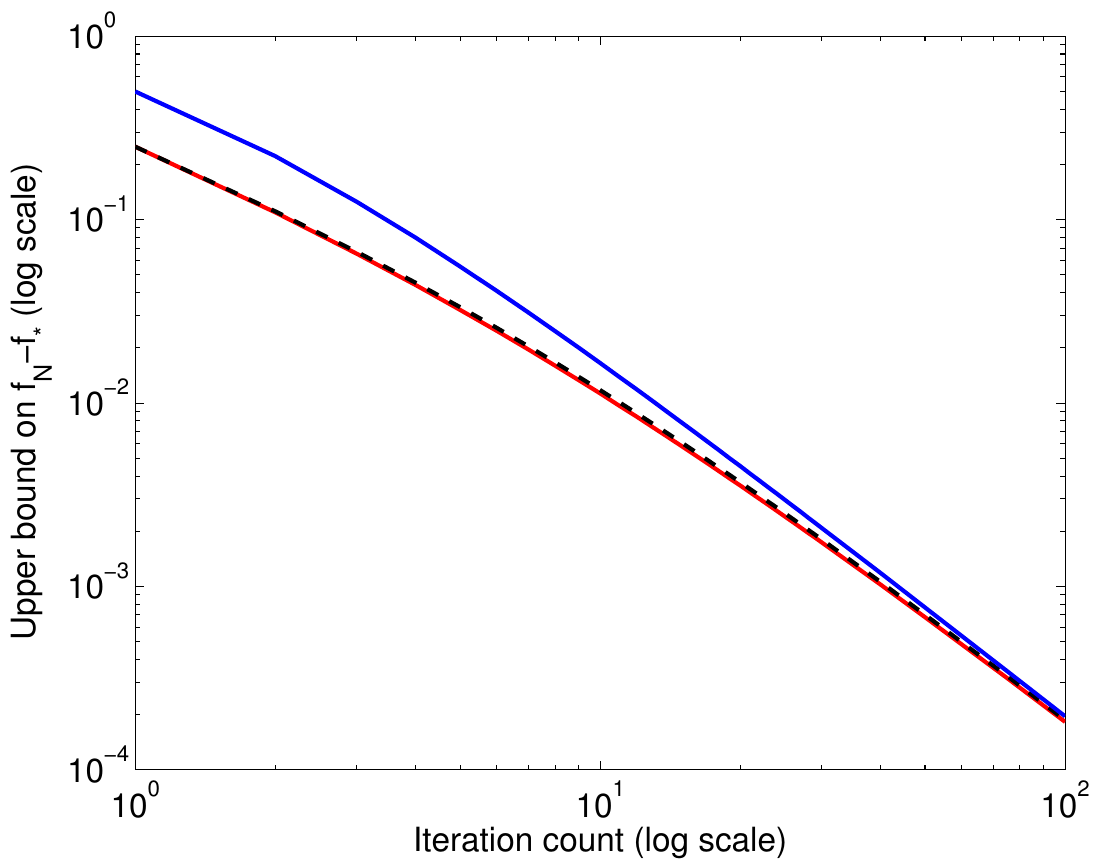}}
\label{Fig:FGMS_comp}
\caption{Comparison of the worst-case convergence speed of the different variants of FPGM1 (left) and FPGM2 (right) \rev{for $N\in\{1,\hdots, 100\}$, $L=1$, and $R=1$}. Curves correspond to the different inertial coefficient, namely, $\alpha^{(a)}_{k}$ (dashed, black) and $\alpha_{k}^{(b)}$ (red), and to the standard guarantee~\eqref{eq:th_FGM} (blue).}
\end{center}
\end{figure}

\subsubsection{New fast proximal gradient methods (FPGM2)} Secondary sequences $\left\{x_k\right\}$ are usually converging slightly faster than primary sequences $\left\{y_k\right\}$ in the unconstrained case ($F^{(2)}=0$), as observed in~\cite{kim2014optimized,taylor2015smooth}. However, some issues may arise with the secondary sequences of FPGM1 when applied to constrained or proximal problems: iterates may in some cases become infeasible, or the objective may become unbounded (see Table~\ref{Tab:PFGM_Conj}). We therefore propose a new variant of a fast proximal gradient method called FPGM2, also with two different step size policies, that does not suffer from theses drawbacks.
\revv{Part of the underlying motivation behind FPGM2 is also the ability to generalize it later to the optimized gradient method.}

\rev{\begin{remark}\label{rem:design} The design of FPGM2 is based on two ideas: on the one hand, it should be equivalent to the standard fast gradient method in the case of smooth unconstrained convex minimization, and on the other hand, it should not move after two consecutive iterates have reached the same optimal point for~\eqref{eq:comp_regsmooth} (i.e., $x_{k-1}=x_{k}=x_*$ implies $x_{k+1}=x_*$). 
\end{remark}}
\vspace{.05cm}
\begin{center}
{\small \fbox{
\parbox{0.9\textwidth}{
        \textbf{Fast Proximal Gradient Method 2 (FPGM2)}
  \begin{itemize}
  \item[] Input: $F^{(1)}\in\mathcal{F}_{0,L}(\E)$, $F^{(2)}\in\mathcal{F}_{0,\infty}(\E)$ $x_0\in\E$, $z_0=y_0=x_0$.\\[-0.2cm]
  \item[] For $k=1:N$\\[-0.5cm]
      \begin{align*}
    &y_{k}=x_{k-1}-\frac{1}{L}B^{-1}\nabla F^{(1)}(x_{k-1})\\
    &z_{k}=y_{k}+\alpha_{k} (y_{k}-y_{k-1})+\frac{\alpha_{k}}{L\gamma_{k-1}}(z_{k-1}-x_{k-1})\\
    &x_{k}=\prox{\gamma_{k} F^{(2)}}{z_{k}}
    \end{align*}
  \end{itemize}
       }}}
\end{center}
\vspace{.3cm}

In this algorithm, we use the coefficients $\gamma_{k}=\frac{\alpha_{k}+1}{L}$. Note that we introduced two intermediate sequences: on the one hand sequence $\left\{\gamma_{k}\right\}_k$, corresponding to the step sizes to be taken by the proximal steps, and on the other hand sequence $\left\{z_{k}\right\}_k$, which keeps track of the subgradient used in the proximal steps (note that $\frac{1}{\gamma_k}(z_k-x_k)$ corresponds to the subgradient used in the proximal step from $z_k$ to $x_k$). Although FPGM2 may look more intricate than the classical FPGM1, it is in fact simpler, as it involves only one sequence on which both implicit (proximal) and explicit (gradient) steps are being taken. Indeed, explicit steps are taken using gradient values of $F^{(1)}$ at $x_k$, and subgradients used in the proximal steps are subgradients of $F^{(2)}$ also at $x_k$. \revv{ This can also be seen by rewriting the iterations of FPGM2 using the secondary sequence $\left\{x_k\right\}_k$ only, in the following way:
\begin{align*}
x_{k+1}= x_k& + \alpha_{k+1} (x_k-x_{k-1}) \\&+\frac{\alpha_{k+1}}{L}B^{-1} \nabla F^{(1)}(x_{k-1}) -\frac{1}{L}B^{-1}\nabla F^{(1)}(x_k)-\frac{\alpha_{k+1}}{L}B^{-1} \nabla F^{(1)}(x_k) \\&+\frac{\alpha_{k+1}}{L}B^{-1}\tilde{\nabla}F^{(2)}(x_{k})-\frac{1}{L}B^{-1}\tilde{\nabla}F^{(2)}(x_{k+1})-\frac{\alpha_{k+1}}{L}B^{-1}\tilde{\nabla}F^{(2)}(x_{k+1}) ,
\end{align*}
with $\tilde{\nabla}F^{(2)}(x_{k})$ the subgradient of $F^{(2)}$ used in the proximal operation generating $x_k$.}

Comparing the different variants of \rev{FPGM2} on \figref{Fig:FGMS_comp} (right plot) leads to the same conclusion as for FPGM1: inertial parameters $\alpha^{(b)}$ perform slightly better than~$\alpha^{(a)}$.  

In Table~\ref{Tab:PFGM_Conj}, we report the different worst-case performance guarantees obtained  numerically for FPGM1 (for both sequences) and FPGM2 (for the better secondary sequence only). We consider three situations: $F^{(2)}=0$ (unconstrained smooth convex minimization), $F^{(2)}\in\mathcal{I}_{\pinf}(\E)$ (constrained smooth convex minimization with projected methods), and $F^{(2)}\in\mathcal{F}_{0,\pinf}(\E)$ (nonsmooth composite convex minimization with proximal methods). 

\begin{table}[ht!]
{
\begin{center}
{\renewcommand{\arraystretch}{1.2}
\begin{tabular}{@{}llll@{}}
\specialrule{2pt}{1pt}{1pt}
Type $\quad$ & \begin{tabular}{@{}c@{}}$F(y_N)-F_*$ \\  (FPGM1)\end{tabular} & \begin{tabular}{@{}c@{}}$F(x_N)-F_*$ \\  (FPGM1)\end{tabular}  & \begin{tabular}{@{}c@{}}$F(x_N)-F_*$ \\  (FPGM2)\end{tabular}  \\
\hline
Unconstrained & \multirow{2}{.2\linewidth}{$\frac{LR^2}{2} \frac{4}{N^2+5N +6}$} & \multirow{2}{.2\linewidth}{$\frac{LR^2}{2} \frac{4}{N^2+7N +4}$} & \multirow{2}{.2\linewidth}{$\frac{LR^2}{2} \frac{4}{N^2+7N +4}$}\\
 ($F^{(2)}=0$)& & & \\[0.15cm]
Constrained & \multirow{2}{.2\linewidth}{$\frac{LR^2}{2} \frac{4}{N^2+5N +2}$} & \multirow{2}{.2\linewidth}{\rev{Infeasible}} & \multirow{2}{.2\linewidth}{$\frac{LR^2}{2} \frac{4}{N^2+7N}$} \\
 ($F^{(2)}\in \mathcal{I}_\infty$)& & & \\[0.15cm]
Non-smooth & \multirow{2}{.2\linewidth}{$\frac{LR^2}{2} \frac{4}{N^2+5N +2}$} & \multirow{2}{.2\linewidth}{\rev{Unbounded}} & \multirow{2}{.2\linewidth}{$\frac{LR^2}{2} \frac{4}{N^2+7N}$}\\
($F^{(2)}\in \mathcal{F}_{0,\infty}$)& & & \\
\specialrule{2pt}{1pt}{1pt}
\end{tabular}
\caption{Worst-case obtained for \rev{FPGM1 and FPGM2} with inertial coefficient $\alpha_k=\frac{k-1}{k+2}$ and $N\geq 1$. }
\label{Tab:PFGM_Conj}}
\end{center}}
\end{table}

All finite convergence results reported in the table actually correspond to specific worst-case functions that we could identify numerically, which means that they provide rigorous lower bounds. After solving the corresponding PEPs numerically (for $L=R=1$ and $1 \le N \le 100$),  we conjecture them to be equal to the exact worst-case guarantees.

We observe that the worst-case guarantees for FPGM2 are slightly better than for FPGM1. 
Guarantees for the unconstrained case are slightly better than those for the constrained and proximal cases, which are equal. Note that the secondary sequence of FPGM1 is not guaranteed to be feasible in the constrained case, and that the corresponding objective value may be unbounded in the proximal case (for any $N \ge 1$).

The worst-case functions identified numerically for the unconstrained case are Huber-shaped functions~\cite{taylor2015smooth}. In the constrained case, we identified one-dimensional linear optimization problems of the form $\min_{x\geq 0} {c}{x}$ as worst-cases, where $c$ is a constant defined by \[ c=\frac{\sqrt{B}R}{2\sum_{j=0}^{N-1}h_{N,j}^{(1)}}\] where $\{h_{N,j}^{(1)}\}$ correspond to the step sizes used in FPGM according to the notation
introduced in (FSLFOM), under the particular choice of $t_{N,N}=1$, and $t_{N,j}=0$ for $0\leq j\leq N-1$). Finally, for the proximal case, our worst-case has function $F^{(1)}(x) = c x$ with the same $c$ as above, and function $F^{(2)}(x)$ may be chosen equal to zero for $x \ge 0$ and to $s x$ for $x < 0$, for any negative value of the slope $s < 0$.

{\subsection{A proximal optimized gradient method} In this section, we consider again the nonsmooth composite convex minimization problem~\eqref{eq:comp_regsmooth}.
In particular, we investigate the possibility of obtaining an optimized method for this setting (i.e., a method whose worst-case performance is the best possible).

Our proposal consists in extending the optimized gradient method (OGM) developed by Kim and Fessler in~\cite{kim2014optimized}, which was originally tailored for smooth unconstrained minimization ($F^{(2)}=0$). In the unconstrained smooth minimization setting, this first-order method \modAT{was recently shown in~\cite{drori2016exact} to have the best achievable worst-case guarantee for the criterion $F_N-F_*$}.

The new method we propose, called POGM, has been obtained by combining ideas obtained from the original OGM~\cite{kim2014optimized} and the nonstandard placement of the proximal operator used for speeding up the convergence of fast proximal gradient methods (FPGM2). It was designed using the same two principles as FPGM2 (see Remark~\ref{rem:design}): \revv{on the one hand, it is equivalent to OGM when applied to smooth unconstrained convex minimization problems, and on the other hand, it remains at an optimal point when it reaches one.}

\vspace{.05cm}
{\small \begin{center}
\fbox{
\parbox{0.9\textwidth}{
        \textbf{Proximal optimized gradient method (POGM)}
  \begin{itemize}
  \item[] Input: $F^{(1)}\in\mathcal{F}_{0,L}(\E)$, $F^{(2)}\in\mathcal{F}_{0,\pinf}(\E)$, $x_0\in\E$, $y_0=x_0$, $\theta_0=1$.\\[-0.2cm]
  \item[] For $k=1:N$\\[-0.5cm]
      \begin{align*}
    &y_{k}=x_{k-1}-\frac{1}{L}B^{-1}\nabla F^{(1)}(x_{k-1}) \\
    &z_{k}=y_{k}+\frac{\theta_{k-1}-1}{\theta_{k}} (y_{k}-y_{k-1})+\frac{\theta_{k-1}}{\theta_{k}} (y_{k}-x_{k-1})+\frac{\theta_{k-1}-1}{L\gamma_{k-1}\theta_{k}}(z_{k-1}-x_{k-1})\\
    &x_{k}=\prox{\gamma_{k}F^{(2)}}{z_{k}}
    \end{align*}
  \end{itemize}
   }}
   \end{center}}  \vspace{.3cm}

In this algorithm, we use the sequence $\gamma_{k}=\frac 1 L \frac{2\theta_{k-1}+\theta_{k}-1}{\theta_{k}}$ and the inertial coefficients proposed in~\cite{kim2014optimized}:
   $$    \theta_{k}=\left\{
    \begin{array}{ll}
    \frac{1+\sqrt{4\theta_{k-1}^2 +1}}{2}, & i\leq N-1,\\
    \frac{1+\sqrt{8\theta_{k-1}^2 +1}}{2}, & i=N.
    \end{array}
    \right.$$
Simply trying to generalize OGM using the standard proximal step on the primary sequence $\left\{y_i\right\}$ (as for FPGM1) does not lead to a converging algorithm. We obtained numerical evidence, i.e., worst-case functions showing that the worst-case bound for this candidate algorithm does not decrease after each iteration (in other words, its worst-case rate is not converging to zero). Therefore we have to introduce the same idea used in FPGM2 concerning the place of the proximal operator. 

We compare POGM to FPGM with inertial coefficients $\alpha^{(b)}_k$  in~\figref{Fig:FISTAvsPOGM}. We obtain worst-case performances about twice better for POGM when compared to both FPGM1 and \rev{FPGM2} between $1$ and $100$ iterations. \rev{Also, we observe that the bound for POGM (equivalent to OGM when $F^{(2)}=0$) is approximately $12\%$ worse than that for OGM~\cite{kim2014optimized} in the worst-case.}

Of course, POGM suffers from the drawback of requiring the knowledge of the number of iterations in advance (because the rule to compute the last coefficient $\theta_{N}$ differs from the rule to compute all the previous ones). This practical disadvantage is not easily solved: if the last $\theta_N$ is updated with the same rule as all the previous coefficients,  performance  is degraded  by a nonnegligible factor, rendering it even slower than FPGM (note that this is already the case for smooth unconstrained minimization~\cite{kim2015convergence}).
\begin{figure}[!ht]
\begin{center}
\includegraphics[scale=0.4]{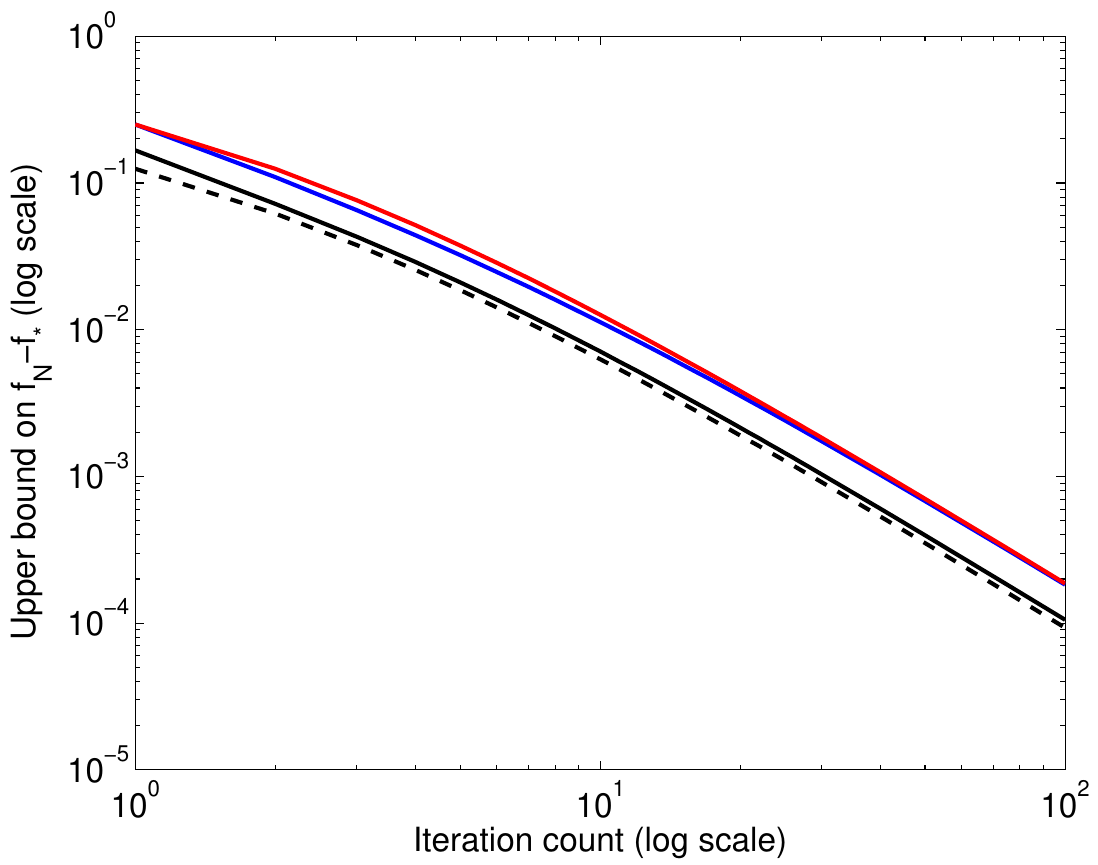}
\caption{Comparison between the worst-case performances of FPGM1 (with inertial coefficients $\alpha^{(b)}_k$) (red), FPGM2 (with inertial coefficients $\alpha^{(b)}_k$) (blue)\rev{, POGM (black), and OGM (dashed, black)} \rev{for $N\in\{1,\hdots, 100\}$, $L=1$, and $R=1$}.}
\label{Fig:FISTAvsPOGM}
\end{center}
\end{figure}}

\subsection{\rev{A conditional gradient method}} \label{ssec:FW_cgm} Consider the constrained smooth convex optimization problem \[
\min_{x\in Q} F(x),\]
with $F\in\mathcal{F}_{0,L}(\E)$ and $Q\subset \E$ a bounded and closed convex set. In that setting, different ways exist for treating the constraint set $Q$. In the previous section, we proposed to use fast gradient methods, which require the ability to project onto the closed convex set $Q$. In this section, we rather consider the standard CGM \rev{(also sometimes referred to as the Frank--Wolfe method)}, which originates from~\cite{frank1956algorithm}. This algorithm has the advantage of avoiding projections onto~$Q$ and performs instead linear optimization on this set (which is typically easier when $Q$ is a polyhedral set).
{\small \begin{center}

\vspace{.05cm}
\fbox{
\parbox{0.9\textwidth}{
        \textbf{Conditional gradient method (CGM)}
  \begin{itemize}
  \item[] Input: $F\in\mathcal{F}_{0,L}(\E)$, closed convex $Q\subset  \E$ with $\normp{x-y}\leq D\ \forall x,y\in Q$, $x_0\in Q$. \\[-0.2cm]
  \item[] For $k=1:N$\\[-0.5cm]
      \begin{align*}
    &y_{k}=\underset{y\in Q}{\text{argmin}}\left\{\inner{\nabla F(x_{k-1})}{y-x_{k-1}}\right\}\\
    &\lambda_{k}= \frac{2}{1+k}\\[0.1cm]
    &\rev{x_{k}}=(1-\lambda_k)\rev{x_{k-1}}+\lambda_k y_{k}
    \end{align*}
  \end{itemize}  
   }}
   \end{center}}
\vspace{.3cm}   

The standard \rev{global convergence guarantee for this method (see e.g.,~\cite[Theorem 1]{jaggi2013revisiting})} is
 \[
F(x_N)-F_*\leq \frac{2LD^2}{N+2},\label{eq:FW_th}
\]
which we compare with the exact bound provided by PEP in \figref{Fig:FW_pep_vs_th} (see \secref{sec:lingramrep}, which shows that CGM  fits into the~\eqref{eq:gen_alg_model} format). The numerical guarantees we obtained by solving the PEP for up to a hundred iterations are between two and three times better than the standard guarantee.

\begin{figure}[!ht]
\begin{center}
\subfigure[Worst-case performance of CGM (red) and its theoretical guarante~\eqref{eq:FW_th} (blue) \rev{for $N\in\{1,\hdots, 100\}$, $L=1$ and $D=1$}.]{\includegraphics[scale=0.3]{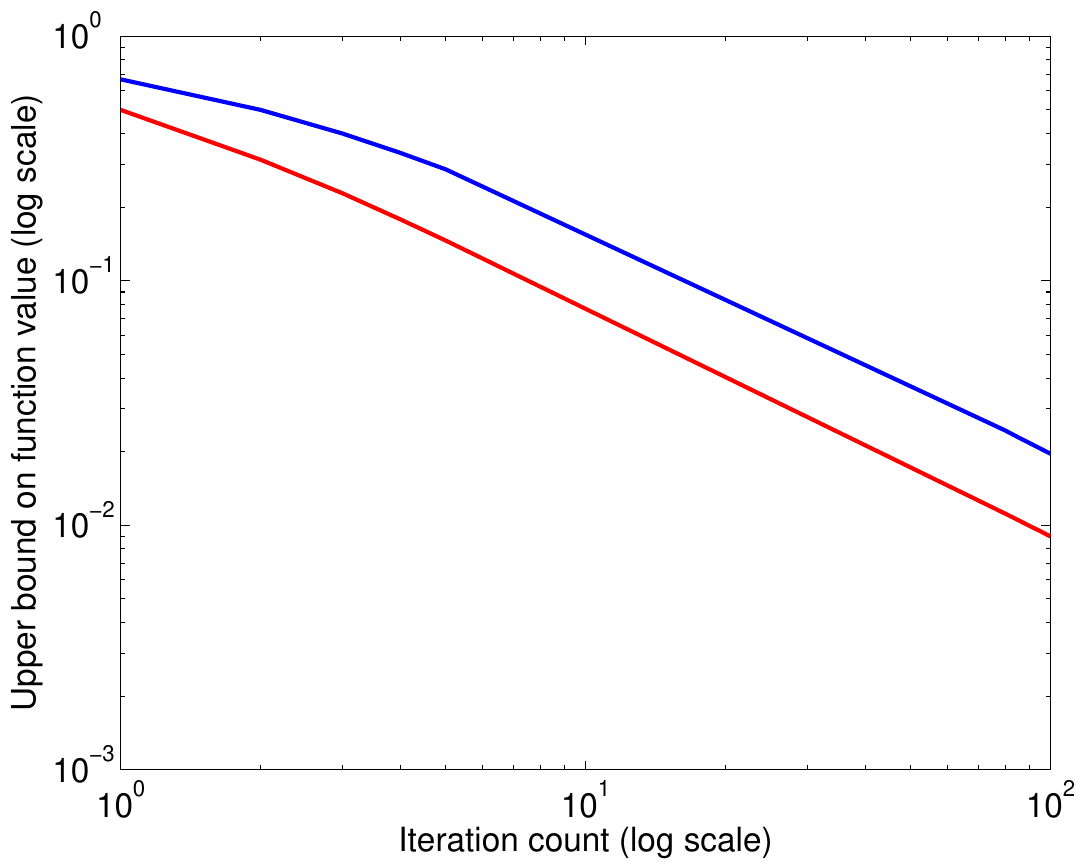}\label{Fig:FW_pep_vs_th}}
\hspace{.1cm}
\subfigure[Worst-case performance of APM (red), DAPM (blue) and lower bound $\frac{R}{\sqrt{N+1}}$ valid for subgradient methods (dashed, black), for \rev{$N\in\{1,\hdots, 100\}$ and $R=1$. 
}. ]{\includegraphics[scale=0.3]{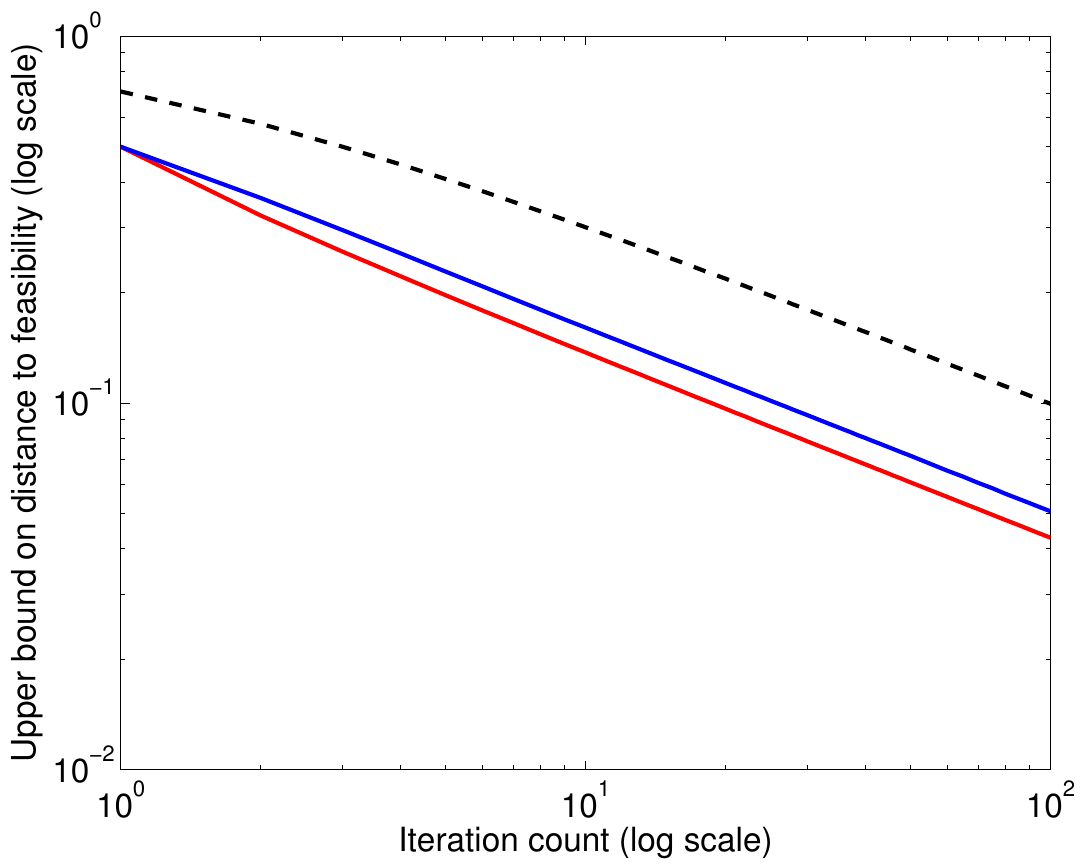}\label{Fig:APMvsDAPM}}
\caption{Numerical analysis of a CGM (left) and of two variants of alternate projections algorithms (right).}
\end{center}
\end{figure}

\subsection{Alternate projection and Dykstra methods}\label{ssec:APM_DAPM}
In this section, we numerically investigate the difference between the worst-case behaviors of the standard alternate projection method (APM) for finding a point in the intersection of two convex sets and the Dykstra~\cite{boyle1986method} method (DAPM) for finding the closest point in the intersection of two convex sets. \revd{APM is a particular instance of subgradient-type descent\footnote{\rev{It can be shown that $\frac{x-\Pi_{Q_k}(x)}{||x-\Pi_{Q_k}(x)||}$ is a subgradient of the function $f(x)$ (at $x$ such that $f(x)=||x-\Pi_{Q_k}(x)||$). Therefore, in the case of two sets $Q_1, Q_2$, and assuming that $x$ is feasible for one of the two sets (say, $Q_1$), a projection onto the other one corresponds to a subgradient step on $f$ with step size $||x-\Pi_{Q_2}(x)||$. Hence, APM is an instance of a subgradient method for $k>1$ (when $x_k$ is feasible for one of the two sets).}} applied to the problem
\begin{equation}
\min_{x\in{\E}}\  \{ f(x)=\max_{i} \normp{x-\Pi_{Q_i}(x)}\},\label{eq:apm}
\end{equation}
whose objective function is convex and nonsmooth (with Lipschitz constant $M=1$)}. Therefore, its expected \rev{global} convergence rate is $\mathcal{O}(\frac{1}{\sqrt{N}})$ (see~\cite[Theorem A.1]{drori2014optimal}). \rev{We compare below the convergence of both APM and DAPM with the standard lower bound for subgradient schemes $\frac{MR}{\sqrt{N+1}}$ as a reference.}
 
 \vspace{.05cm}
{\small \begin{center}

\fbox{
\parbox{0.7\textwidth}{
        \textbf{Alternate projection method (APM)}
  \begin{itemize}
  \item[] Input: $x_0\in\E$, convex sets $Q_1, Q_2\subseteq \E$, $\normp{x_0-x_*}\leq R$, for some $x_*\in Q_1\cap Q_2$. \\[-0.2cm]
  \item[] For $k=1:N$\\[-0.5cm]
      \begin{align*}
    &x_{k}=\Pi_{Q_2}(\Pi_{Q_1}(x_{k-1}))
    \end{align*}
  \end{itemize}  
   }}
   \end{center}}
   {\small \begin{center}

\fbox{
\parbox{0.7\textwidth}{
        \textbf{Dykstra alternate projection method (DAPM)}
  \begin{itemize}
  \item[] Input: $x_0\in\E$, convex sets $Q_1, Q_2\subseteq \E$, $\normp{x_0-x_*}\leq R$, for some $x_*\in Q_1\cap Q_2$. Initialize $p_0=q_0=0$. \\[-0.2cm]
  \item[] For $k=0:N-1$\\[-0.5cm]
      \begin{align*}
    &y_{k}=\Pi_{Q_1}(x_k+p_k)\\
    &p_{k+1}=x_k+p_k-y_k\\
    &x_{i+1}=\Pi_{Q_2}(y_k+\rev{q_k})\\
    &q_{k+1}=y_k+q_k-x_{k+1}
    \end{align*}
  \end{itemize}  
   }}
   \end{center}}
   \vspace{.3cm}
   
The \rev{performance measure} used is $\min_{x\in Q_1} \normp{x-x_N}= \normp{x_N - \Pi_{Q_1}(x_N)}$ (noting that $x_N\in Q_2$ always holds). We do not provide details on the corresponding PEP here, as it is very similar to the previous sections. The results for APM and DAPM are shown in~\figref{Fig:APMvsDAPM}, where the (expected) convergence in $\mathcal{O}(\frac{1}{\sqrt{N}})$ is clearly obtained. Interestingly, DAPM converges slightly slower than APM (more precisely, DAPM has a worst-case about $18\%$ larger
than APM), which is therefore more advisable for finding a point in the intersection of two convex sets (in terms of worst-case performance, when no additional structure is assumed). In addition, note that both APM and DAPM have a worst-case which is about twice  better than the standard lower bound for explicit nonsmooth schemes.

\section{Conclusion}
In this work, we presented a performance estimation approach to analyze first-order algorithms for composite optimization problems. The results of~\cite{taylor2015smooth} were largely extended to handle both larger classes of (composite) objective functions and larger classes of first-order algorithms (also in a more general setting for handling pairs of conjugate norms). 

Our contribution was essentially threefold: first, we developed specific interpolation conditions for different classes of convex and nonconvex functions; then, we exploited those interpolation conditions to formulate the exact worst-case problem for fixed-step linear first-order methods, and finally we applied that methodology to provide tight analyses for different first-order methods. Among others, we presented a new analytical guarantee for the proximal point algorithm that is twice better than previously known and improved the standard \rev{worst-case} guarantee for the conditional gradient method by more than a factor of two. On the way, we also proposed an extension of the optimized gradient method proposed by Kim and Fessler~\cite{kim2014optimized} that incorporates a projection or a proximal operator.

As further research, we believe this methodology should be applied to refine analyses of methods fitting in the context of fixed-step linear first-order methods, and possibly extended to handle dynamic step size rules. To this end, a possibility is to explore convex relaxations of the resulting possibly nonconvex performance estimations problems. As an example, we believe it would be interesting to analyze algorithms involving line-search, such as backtracking or Armijo--Wolfe procedures (a first step in that direction is taken in~\cite{de2016worst}, which study the worst-case behavior of steepest descent with exact line-search). Moreover, it seems to us that the performance estimation approach could be used to \rev{refine the analyses of randomized coordinate descent-type algorithms~\cite{nesterov2012efficiency}}. Performance estimation problems also opened the door for looking toward optimized methods, as proposed by Kim and Fessler~\cite{kim2014optimized} for unconstrained smooth convex minimization.

Finally, algorithmic \revv{analyses} using performance estimation problems are intrinsically limited by our ability to solve semidefinite problems, both numerically (when the number of iterations is large) or analytically (to obtain results valid for any number of iterations). Therefore, any idea leading to (convex) programs that are easier to solve while maintaining reasonable guarantees would be very \rev{advantageous}.

\textbf{Software.} An easy-to-use MATLAB implementation of the approach is available
at \url{https://github.com/AdrienTaylor/Performance-Estimation-Toolbox}.

\newpage
\bibliographystyle{siamplain}
\bibliography{bib_}
\appendix
\section{Proof of upper bound in Theorem~\ref{thm:PPA_conv}}
\label{sec:app_proof_thm_ppa1} 
In order to express the corresponding PEP in the simplest form, we heavily rely on some straightforward simplifications of~\eqref{PEPS:SDP} (see Corollary~\ref{cor:Large_scale_FSLFOM} and Remark~\ref{rem:simp_SDP}). Let us denote by $P_N$ the matrix containing the information harvested after $N$ iterations: $P_N=[g_1 \ g_2 \ \hdots \ g_N \ Bx_0]$  (we use the notation $g_i$ for subgradients $g_i\in\partial F(x_i)$), and by $G_N$ its corresponding Gram matrix (see \secref{sec:GramRep}). Also, we introduce the step size vectors $m_k$ that express each iterate $x_k$ in terms of $x_0$ and the subgradients $\{ g_i \}_{1 \le i \le N}$, that is $x_{k}=P_N m_{k}\ (k=0,\hdots,N).$ Using the standard notation $e_i$ for the unit vector having a single $1$ as its $i$th component, this results in the following explicit expressions for $m_k$: $m_k=e_{N+1} -  \sum_{i=1}^{k} \alpha_i e_i,$ along with $m_0=e_{N+1}$ and $m_*=0$ (where we assumed without loss of generality that $x_*=0$). 

In order to perform the worst-case analysis for PPA, we now formulate the performance estimation problem~\eqref{Intro:dPEP} as the following SDP, the simplified version of~\eqref{PEPS:SDP} where the $x_k$'s ($k=1,\hdots,N$) have been substituted using the equation defining the iterates $x_{k}=x_{k-1}-\alpha_{k}B^{-1}g_{k}$:
\begin{align*}
\max_{G_N\in\mathbb{S}^{N+1}, f_1,\hdots,f_N,f_*\in\mathbb{R}^{N}}  f_N-f_* \tag{PPA-PEP}\label{PEPS:PPA}, \
\text{s.t.  } f_j-f_i + \text{Tr}(A_{ij}G_N) &\leq 0, \quad i,j\in \left\{1,\hdots,N,*\right\}\\
 \normpsq{x_0-x_*} &\leq R^2,\\
 G_N&\succeq 0,
\end{align*}
with matrices $2A_{ij}=e_j(m_i-m_j)^{\top\!}+(m_i-m_j)e_j^{\top\!}$ (where $e_*=0$) coming from the nonsmooth convex interpolation inequalities (see condition~\eqref{eq:ns_cvx_interp}).
In order to obtain an analytical upper bound for PPA, we consider the Lagrangian dual to~\eqref{PEPS:PPA}, which is given by the following:
\begin{align*}
\min_{\lambda_{ij}\geq 0, \tau\geq 0}  \tau R^2 \tag{PPA-dPEP}\label{PEPS:dPPA}
\ \text{s.t.  } e_{N}-\sum_{i}\sum_{j\neq i} (\lambda_{ij}-\lambda_{ji})e_j&=0,\\ \sum_{i}\sum_{j\neq i} \lambda_{ij}A_{ij}+\tau m_0 m_0^{\top\!}&\succeq 0
\end{align*}
\rev{(where the constraint corresponding to $f_*$ can be discarded since it is clear that letting $f_*=0$ does not change the optimal solution of~\eqref{PEPS:PPA}).}
Note that the set of equality constraints can be assimilated to a set of \emph{flow} constraints on a complete directed graph. That is, considering a graph where the optimum and each iterate correspond to nodes, each nonnegative $\lambda_{ij}$ corresponds to the flow on the edge going from node $j$ to node $i$ (we choose this direction by convention). This flow constraint imposes that the outgoing flow equals the ingoing flow for every node, except at the node for final iterate $N$, where the outgoing flow should be equal to $1$, and at the optimum node, where the incoming flow should be equal to $1$.
We show that the following choice is a feasible point of the dual~\eqref{PEPS:dPPA}.
{\begin{align*}
&\lambda_{i,i+1}=\frac{\sum_{k=1}^i \alpha_k}{2\sum_{k=1}^N \alpha_k-\sum_{k=1}^i \alpha_k},\quad &i\in\left\{1,\hdots,N-1\right\},\\
&\lambda_{*,i}=\frac{2\alpha_i \sum_{k=1}^N \alpha_k}{\left(2\sum_{k=1}^N \alpha_k-\sum_{k=1}^i \alpha_k\right)\left(2\sum_{k=1}^N \alpha_k-\sum_{k=1}^{i-1} \alpha_k\right)},\quad &i\in\left\{1,\hdots,N\right\},\\
&\tau=\frac{1}{4\sum_{k=1}^N \alpha_k},
\end{align*}}and $\lambda_{ij}=0$ otherwise. First, we clearly have $\lambda_{ij}\geq 0$ and some basic computations allow us to verify that the equality constraints from~\eqref{PEPS:dPPA} are satisfied:
\begin{align*}
\lambda_{*,1}-\lambda_{1,2} =0, \ \lambda_{*,i}+\lambda_{i-1,i}-\lambda_{i,i+1} =0 \ (i\in\{2,\hdots,N-1\}), \ \lambda_{*,N}+\lambda_{N-1,N} =1.
\end{align*}
It remains to show that the corresponding dual matrix $S$ is positive semidefinite.
{\begin{align*}
2S=\sum_{i=1}^{N-1}&2\alpha_{i+1}\lambda_{i,i+1}e_{i+1}e_{i+1}^{\top\!} +2\tau e_{N+1}e_{N+1}^{\top\!}\\&+\sum_{i=1}^{N}\lambda_{*,i}\left[e_i\left(-e_{N+1}+\sum_{k=1}^i \alpha_ke_k\right)^{\top}+\left(-e_{N+1}+\sum_{k=1}^i \alpha_ke_k\right)e_i^{\top}\right].
\end{align*}}In order to reduce the number of indices to be used, we will note  $\lambda_i=\lambda_{i,i+1}$ and $\mu_i=\lambda_{*,i}$. Then, using the equality constraints, we arrive at the following dual matrix:
{\small \begin{align*}
2S= 
\begin{pmatrix}
2\alpha_1\lambda_1 & \alpha_1\mu_2 & \alpha_1\mu_3 & \hdots & \alpha_1\mu_{N-1} & \alpha_1\mu_N & -\mu_1\\
\alpha_1\mu_2 & 2\alpha_2\lambda_2 & \alpha_2\mu_3 & \hdots & \alpha_2\mu_{N-1} & \alpha_2\mu_N & -\mu_2\\
\alpha_1\mu_3 & \alpha_2\mu_3 & 2\alpha_3\lambda_3 & \hdots & \alpha_3\mu_{N-1} & \alpha_3\mu_N & -\mu_3\\
\vdots &   & \ddots & \ddots &  & \vdots & \vdots\\
\alpha_1\mu_{N-1} & \alpha_2\mu_{N-1} & \alpha_3\mu_{N-1} & \hdots & 2\alpha_{N-1}\lambda_{N-1} & \alpha_{N-1}\mu_N & -\mu_{N-1}\\
\alpha_1\mu_N & \alpha_2\mu_N & \alpha_3\mu_N & \hdots  & \alpha_{N-1}\mu_{N} & 2\alpha_N & -\mu_N\\
-\mu_1 & -\mu_2 & -\mu_3 & \hdots & -\mu_{N-1} & -\mu_N & 2\tau\\
\end{pmatrix}.
\end{align*}}In order to prove $S\succeq 0$, we first use a Schur complement and then show that the resulting matrix is diagonally dominant with positive diagonal elements. After taking the Schur complement (with respect to the lower right scalar component $2\tau$), we obtain the matrix~$\tilde{S}$:
{\small \begin{align*}
\tilde{S}= 
\begin{pmatrix}
2\alpha_1\lambda_1 & \alpha_1\mu_2 & \alpha_1\mu_3 & \hdots & \alpha_1\mu_{N-1} & \alpha_1\mu_N \\
\alpha_1\mu_2 & 2\alpha_2\lambda_2 & \alpha_2\mu_3 & \hdots & \alpha_2\mu_{N-1} & \alpha_2\mu_N \\
\alpha_1\mu_3 & \alpha_2\mu_3 & 2\alpha_3\lambda_3 & \hdots & \alpha_3\mu_{N-1} & \alpha_3\mu_N \\
\vdots &   & \ddots & \ddots &  & \vdots \\
\alpha_1\mu_{N-1} & \alpha_2\mu_{N-1} & \alpha_3\mu_{N-1} & \hdots & 2\alpha_{N-1}\lambda_{N-1} & \alpha_{N-1}\mu_N \\
\alpha_1\mu_N & \alpha_2\mu_N & \alpha_3\mu_N & \hdots  & \alpha_{N-1}\mu_{N} & 2\alpha_N
\end{pmatrix}
-\frac{1}{2\tau}
\begin{pmatrix}
\mu_1\\
\mu_2\\
\vdots\\
\mu_{N}
\end{pmatrix}
\begin{pmatrix}
\mu_1\\
\mu_2\\
\vdots\\
\mu_{N}
\end{pmatrix}^{\top\!}.
\end{align*}}The first step to show the diagonally dominant character of $\tilde{S}$ is to note that every nondiagonal element of $\tilde{S}$ is nonpositive: $\alpha_j\mu_i-\frac{\mu_i\mu_j}{2\tau}\leq 0 \ \forall i\neq j.$
Indeed, this is equivalent to writing this in the following form ($\mu_i> 0$):
{\small \begin{align*}
&\alpha_j-\frac{\mu_j}{2\tau}=\alpha_j\left( \frac{\left(2\sum_{k=1}^N \alpha_k-\sum_{k=1}^i \alpha_k\right)\left(2\sum_{k=1}^N \alpha_k-\sum_{k=1}^{i-1} \alpha_k\right)-\left(2\sum_{k=1}^N \alpha_k\right)^2}{\left(2\sum_{k=1}^N \alpha_k-\sum_{k=1}^i \alpha_k\right)\left(2\sum_{k=1}^N \alpha_k-\sum_{k=1}^{i-1} \alpha_k\right)}\right)\leq 0,
\end{align*}}since $\alpha_k\geq 0$ by assumption. This allows us to discard the absolute values in the diagonal dominance criteria. Then, using the equality constraints, we obtain an expression for the sum of all nondiagonal elements of line $i$ of $\tilde{S}$:
{\begin{align*}
\mu_i\sum_{j=1}^{i-1}\alpha_j + \alpha_i\sum_{j=i+1}^{N}\mu_j&-\frac{\mu_i}{2\tau}\sum_{j\neq i}\mu_j\\&=
 \left\{\begin{array}{ll}
\mu_i \sum_{j=1}^{i-1}\alpha_j + \alpha_i(1-\lambda_i)-\frac{1}{2\tau}\mu_i(1-\mu_i), & \text{ if }i<N, \\
\mu_N \sum_{j=1}^{N-1}\alpha_j-\frac{1}{2\tau}\mu_N(1-\mu_N) & \text{ if }i=N.
\end{array}\right.
\end{align*}}Using the values of $\mu_i$, $\lambda_i$, and $\tau$ along with elementary computations allows to verify that $\forall i\in\left\{1,\hdots,N\right\}$,
\begin{align*}
\left\{\begin{array}{lll}
-(\mu_i \sum_{j=1}^{i-1}\alpha_j + \alpha_i(1-\lambda_i)-\frac{1}{2\tau}\mu_i(1-\mu_i))&=2\alpha_i\lambda_i-\frac{\mu_i^2}{2\tau} & \text{if } i=1,\hdots,N-1, \\
-(\mu_i \sum_{j=1}^{i-1}\alpha_j-\frac{1}{2\tau}\mu_i(1-\mu_i))&=2\alpha_i-\frac{\mu_i^2}{2\tau} & \text{if } i=N,
\end{array}\right.
\end{align*}
which implies diagonal dominance of $\tilde{S}$ (even more: the sum of the elements of each line equals $0$). \qed
\end{document}

%% file: ex_shared.tex
\usepackage{amsfonts}
\usepackage{graphicx}
\usepackage{epstopdf}
\usepackage{algorithmic}
\usepackage{subfigure}
\ifpdf
  \DeclareGraphicsExtensions{.eps,.pdf,.png,.jpg}
\else
  \DeclareGraphicsExtensions{.eps}
\fi

\newcommand{\TheTitle}{Worst-case performance of first-order methods} 
\newcommand{\TheAuthors}{A.B.~Taylor, J.M.~Hendrickx, F.~Glineur}

\headers{\TheTitle}{\TheAuthors}

\title{Exact Worst-case Performance of First-order Methods for Composite Convex Optimization\thanks{This paper
presents research results of the Belgian Network DYSCO (Dynamical Systems, Control, and Optimization), funded by the Interuniversity Attraction Poles Programme, initiated by the Belgian State, Science Policy Office, and of the Concerted Research Action (ARC) programme supported by the Federation Wallonia-Brussels (contract ARC 14/19-060). The scientific responsibility rests with its authors. The first author is a FRIA fellow.}
}
\author{
  Adrien B. Taylor\thanks{Universit\'e catholique de Louvain, Information and Communication Technologies, Electronics and Applied Mathematics (ICTEAM) Institute, B-1348 Louvain-la-Neuve, Belgium (adrien.taylor@uclouvain.be, julien.hendrickx@uclouvain.be).}
  \and
  Julien M. Hendrickx\footnotemark[2]
  \and
  Fran\c cois Glineur\thanks{Universit\'e catholique de Louvain, Center for Operations Research and Econometrics (CORE),
  	B-1348 Louvain-la-Neuve, Belgium, and Universit\'e catholique de Louvain, Information and Communication Technologies, Electronics and Applied Mathematics (ICTEAM) Institute, B-1348 Louvain-la-Neuve, Belgium (francois.glineur@uclouvain.be).}
}

\usepackage{amsopn}


%% file: preamble.tex
\ifpdf
\hypersetup{
  pdftitle={\TheTitle},
  pdfauthor={\TheAuthors}
}
\fi
\DeclareMathOperator{\dom}{dom}

\newcommand{\figref}[1]{\figurename~\ref{#1}}
\newcommand{\secref}[1]{section~\ref{#1}}

\newcommand{\trace}[1]{\text{Tr}\left(#1\right)}
\newcommand{\argmin}[1]{\underset{#1}{\text{argmin }}}

\newcommand{\prox}[2]{{\mathrm{prox}}_{#1}\left(#2\right)}

\newcommand{\pinf}{\infty}
\newcommand{\cinf}{\cup\left\{\pinf\right\}}
\newcommand{\R}{\mathbb{R}}
\newcommand{\Rinf}{\mathbb{R}\cinf}

\newcommand{\E}{\mathbb{E}}
\newcommand{\Es}{\mathbb{E}^*}

\newcommand{\normstd}[1]{{\left\lVert#1\right\rVert}}

\newcommand{\normp}[1]{{\left\lVert#1\right\rVert_{\E}}}
\newcommand{\normd}[1]{{\left\lVert#1\right\rVert_{\Es}}}

\newcommand{\normdsqsm}[1]{{\lVert#1\rVert_{\Es}^2}}
\newcommand{\normpsq}[1]{{\left\lVert#1\right\rVert_{\E}^2}}
\newcommand{\normdsq}[1]{{\left\lVert#1\right\rVert_{\Es}^2}}

\newcommand{\inner}[2]{{\left< #1, #2\right>}}
\newcommand{\innerS}[2]{{\langle #1, #2 \rangle}}
\newcommand{\innerE}[2]{{\langle #1, #2 \rangle}_{\E}}
\newcommand{\innerEs}[2]{{\langle #1, #2 \rangle}_{\Es}}

\newcommand{\inneraEs}[2]{{\left< #1, #2 \right>}_{\Es}}

\newcommand{\fdef}{\E\rightarrow\Rinf}
\newcommand{\fsdef}{\Es\rightarrow\Rinf}

\newcommand{\modAT}[1]{{#1}}
